\documentclass[dvipdfmx,a4paper,12pt]{article}

\usepackage{amsmath,amssymb,amsthm,mathtools,hyperref,multirow}

\hypersetup{
setpagesize=false,
bookmarksnumbered=true,
bookmarksopen=true,
colorlinks=true,
linkcolor=blue,
citecolor=red
}

\newcommand{\N}{\mathbb{N}}
\newcommand{\R}{\mathbb{R}}
\newcommand{\C}{\mathcal{C}}
\newcommand{\ep}{\varepsilon}
\newcommand{\pa}{\partial}
\newcommand{\II}{I\hspace{-1.2pt}I}
\newcommand{\III}{I\hspace{-1.2pt}I\hspace{-1.2pt}I}
\newcommand{\supu}{\overline{u}}
\newcommand{\supv}{\overline{v}}
\newcommand{\lr}[1]{\langle#1\rangle}
\DeclareMathOperator{\supp}{supp}

\newtheorem{theorem}{Theorem}[section]
\newtheorem{lemma}[theorem]{Lemma}

\theoremstyle{remark}
\newtheorem{remark}{Remark}[section]
\theoremstyle{definition}

\newtheorem*{assumption}{Assumption}
\numberwithin{equation}{section}

\title{\Large\textbf{The critical Fujita exponent for one-dimensional semilinear heat equations with potentials and space-dependent nonlinearities}}
\author{
Reiri Miyamoto\footnote{
Department of Mathematics,
Faculty of Science and Technology,
Tokyo University of Science,
2641 Yamazaki,
Noda-shi,
Chiba,
278-8510,
Japan,
E-mail: \texttt{rmiyamoto2001@outlook.jp}
}
\and
Motohiro Sobajima\footnote{
Department of Mathematics,
Faculty of Science and Technology,
Tokyo University of Science,
2641 Yamazaki,
Noda-shi,
Chiba,
278-8510,
Japan,
E-mail: \texttt{msobajima1984@gmail.com}
}
}
\date{}

\begin{document}

\maketitle

\newenvironment{summary}{\vspace{.5\baselineskip}\begin{list}{}{
\setlength{\baselineskip}{0.85\baselineskip}
\setlength{\topsep}{0pt}
\setlength{\leftmargin}{12mm}
\setlength{\rightmargin}{12mm}
\setlength{\listparindent}{0mm}
\setlength{\itemindent}{\listparindent}
\setlength{\parsep}{0pt}
\item\relax
}
}
{\end{list}\vspace{.5\baselineskip}}

\begin{summary}
{\footnotesize\textbf{Abstract.}
This paper is concerned with the existence/nonexistence of nontrivial global-in-time solutions to the Cauchy problem
\begin{equation}
	\begin{cases}
		\label{P}
		\tag{P}
		\pa_tu-\pa_x^2u+Vu=(1+x^2)^{-\frac{m}{2}}u^p,&x\in\R,\ t>0,\\
		u(x,0)=u_0(x)\ge0,&x\in\R,
	\end{cases}
\end{equation}
where $p>1$, $m\ge0$, $u_0\in BC(\R)$ and the potential $V=V(x)\in BC(\R)$ satisfies a certain property.
More precisely, we determine the critical Fujita exponent for \eqref{P}, that is, the threshold for the global existence/nonexistence of \eqref{P}.
}
\end{summary}



\tableofcontents

\newpage

\section{Introduction}
\label{section1}
In this paper, we consider the Cauchy problem of one-dimensional semilinear heat equations of the form
\begin{equation}
	\label{main eq}
	\begin{cases}
		\pa_tu-\pa_x^2u+Vu=\lr{x}^{-m}u^p,&x\in\R,\ t>0,\\
		u(x,0)=u_0(x)\ge0,&x\in\R,
	\end{cases}
\end{equation}
where $p>1$, $m\ge0$, $\lr{x}=\sqrt{1+x^2}$, $u_0\in BC(\R)$ and the potential $V=V(x)\in BC(\R)$ satisfies a certain property specified below.
Throughout the present paper, we say that $u$ is a global-in-time solution to \eqref{main eq} if $u\in C^{2;1}(\R\times(0,\infty))\cap C(\R\times[0,\infty))$ and $u$ satisfies \eqref{main eq} in $\R\times(0,\infty)$.

The aim of the present paper is to discuss the global existence/nonexistence of \eqref{main eq}, that is, the existence/nonexistence of nontrivial global-in-time solutions to \eqref{main eq}.
Our interest is how the potential $V$ and the space-dependent weight $\lr{x}^{-m}$ contained in the nonlinearity affect the critical phenomenon for the global existence/nonexistence of \eqref{main eq}.

In Fujita \cite{F1966}, the following Cauchy problem was considered:
\begin{equation}
	\label{Fujita eq}
	\begin{cases}
		\pa_tu-\Delta u=u^p,&x\in\R^N,\ t>0,\\
		u(x,0)=u_0(x)\ge0,&x\in\R^N,
	\end{cases}
\end{equation}
where $N\in\N$.
He proved the following:
\begin{enumerate}
	\item If $1<p<1+\frac{2}{N}$, then \eqref{Fujita eq} does not have nontrivial global-in-time solutions;
	\item If $p>1+\frac{2}{N}$, then \eqref{Fujita eq} possesses a nontrivial global-in-time solution for some initial data.
\end{enumerate}
The threshold $p_\text{F}(N)=1+\frac{2}{N}$ for the global existence/nonexistence of \eqref{Fujita eq} is called the critical Fujita exponent for \eqref{Fujita eq}.
In subsequent papers by Hayakawa \cite{H1973}, Sugitani \cite{S1975} and Kobayashi--Sirao--Tanaka \cite{KST1977}, it was shown that the statement \textbf{(i)} also holds for the critical case $p=p_\text{F}(N)$.
After that, similar results for various nonlinear evolution equations have been studied by many mathematicians (for details, see e.g. Levine \cite{L1990}, Deng--Levine \cite{DL2000}, Quittner--Souplet \cite{QS2007} and references therein).

In Pinsky \cite{P1997}, the following Cauchy problem with space-dependent nonlinearities was considered:
\begin{equation}
	\label{Pinsky eq}
	\begin{cases}
		\pa_tu-\Delta u=bu^p,&x\in\R^N,\ t>0,\\
		u(x,0)=u_0(x)\ge0,&x\in\R^N,
	\end{cases}
\end{equation}
where $b=b(x)$ is a nonnegative H\"{o}lder continuous function satisfying $b_1|x|^l\le b(x)\le b_2|x|^l$ for some constants $l\in\R$, $b_1,b_2>0$ and sufficiently large $|x|$.
He determined the critical Fujita exponent $p_\text{P}(N,l)$ for \eqref{Pinsky eq} as follows:
\[
p_\text{P}(N,l)=
\begin{cases}
	\max\{2,3+l\},&\text{if}\ N=1,\\
	1+\frac{[2+l]_+}{N},&\text{if}\ N\ge2.
\end{cases}
\]
More precisely, he proved the following:
\begin{enumerate}
	\item If $1<p\le p_\text{P}(N,l)$, then \eqref{Pinsky eq} does not have nontrivial global-in-time solutions;
	\item If $p>p_\text{P}(N,l)$, then \eqref{Pinsky eq} possesses a nontrivial global-in-time solution for some initial data.
\end{enumerate}
In particular, if $N=1$ and $1<p\le2$, then the global existence of \eqref{Pinsky eq} is always false even if $l$ approaches $-\infty$.
This phenomenon can be explained by the recurrence of the one-dimensional Brownian motion.

Critical Fujita exponents for semilinear heat equations with potentials were first studied by Zhang \cite{Z2001} and subsequently studied by Ishige \cite{I2008} and Pinsky \cite{P2009}.
Recently, in Ishige--Kawakami \cite{IK2020}, the following Cauchy problem was considered:
\begin{equation}
	\label{Ishige--Kawakami eq}
	\begin{cases}
		\pa_tu-\Delta u+Wu=u^p,&x\in\R^N,\ t>0,\\
		u(x,0)=u_0(x)\ge0,&x\in\R^N,
	\end{cases}
\end{equation}
where the potential $W=W(x)$ satisfies
\begin{equation}
	\label{ass on W}
	\begin{cases}
		\textbf{(i)}&W=W(r)\in C^1([0,\infty));\\
		\textbf{(ii)}&\lim_{r\to\infty}r^\theta|r^2W(r)-\lambda|=0\quad\text{for some}\ \theta>0,\ \lambda\in\R;\\
		\textbf{(iii)}&\sup_{r\ge1}|r^3W^\prime(r)|<\infty.
	\end{cases}
\end{equation}
In particular, in \cite{IK2020}, the critical Fujita exponent was observed via the notion in the criticality theory for Schr\"{o}dinger operators (cf. Remark \ref{criticality theory}).
The asymptotic behavior of the unique (positive) solution $U$ to the Cauchy problem of the linear ordinary differential equation
\[
\begin{cases}
		U^{\prime\prime}+\frac{N-1}{r}U^\prime-W(r)U=0,&r>0,\\
		U(0)=1,\ U^\prime(0)=0
\end{cases}
\]
can be classified as
\[
U(r)
\begin{cases}
	\sim U_0r^{\alpha_\lambda^+},&\text{if}\ S\ \text{is subcritical},\ \lambda>\lambda_*(N),\\
	\sim U_0r^{\alpha_\lambda^+}\log(1+r),&\text{if}\ S\ \text{is subcritical},\ \lambda=\lambda_*(N),\\
	\sim U_0r^{\alpha_\lambda^-},&\text{if}\ S\ \text{is critical},\\
	\text{does not exist},&\text{if}\ S\ \text{is supercritical} 
\end{cases}
\]
as $r\to\infty$, where $U_0>0$, $\lambda_*(N)=-\frac{(N-2)^2}{4}$ and $\alpha_\lambda^\pm$ are the roots of the algebraic equation $\alpha(\alpha+N-2)=\lambda$ with $\alpha_\lambda^-\le\alpha_\lambda^+$, that is, $\alpha_\lambda^\pm=\frac{-(N-2)\pm\sqrt{(N-2)^2+4\lambda}}{2}$, see Murata \cite{M1986} (also note that if $W$ satisfies \eqref{ass on W} and $S$ is nonnegative, then $\lambda\ge\lambda_*(N)$, see e.g. \cite{P2009}).
Using the asymptotic profile of $U$, Ishige--Kawakami \cite{IK2020} determined the critical Fujita exponent for \eqref{Ishige--Kawakami eq} as follows:
\[
p_\text{IK}(N,\lambda)=
\begin{cases}
	1+\frac{2}{N+\alpha_\lambda^+},&\text{if}\ S\ \text{is subcritical},\\
	1+\frac{2}{N+\alpha_\lambda^-},&\text{if}\ S\ \text{is critical},\ \alpha_\lambda^-\ge\alpha_*(N),\\
	\frac{2}{N+2\alpha_\lambda^-},&\text{if}\ S\ \text{is critical,}\ -\frac{N}{2}<\alpha_\lambda^-<\alpha_*(N),\\
	\infty,&\text{if}\ S\ \text{is critical},\ \alpha_\lambda^-<-\frac{N}{2},\\
	\infty,&\text{if}\ S\ \text{is supercritical},
\end{cases}
\]
where $\alpha_*(N)\in\left(-\frac{N}{2},-\frac{N-2}{2}\right)$ is the real number satisfying $1+\frac{2}{N+\alpha_*(N)}=\frac{2}{N+2\alpha_*(N)}$, that is, $\alpha_*(N)=\frac{-(3N+2)+\sqrt{N^2+12N+4}}{4}$.
More precisely, they proved the following:
\begin{enumerate}
	\item If $1<p<p_\text{IK}(N,\lambda)$, then \eqref{Ishige--Kawakami eq} does not have nontrivial global-in-time solutions;
	\item If $p>p_\text{IK}(N,\lambda)$, then \eqref{Ishige--Kawakami eq} possesses a nontrivial global-in-time solution for some initial data.
\end{enumerate}
Roughly speaking, their strategy is to employ the heat kernel estimate developed in Grigor'yan--Saloff-Coste \cite{GS2005}; note that global existence/nonexistence in some threshold cases including $\alpha_\lambda^-=-\frac{N}{2}$ has been left open.

The purpose of the present paper is to discuss the global existence/nonexistence of the one-dimensional problem \eqref{main eq} including the threshold case as mentioned above. In particular, to find the phenomenon also in the threshold case, we shall employ an alternative approach which is explained later and different from Ishige--Kawakami \cite{IK2020} (the use of heat kernels).

Now we are in a position to state the main result of the present paper.
We assume that $V$ satisfies the following conditions.

\begin{assumption}
	There exists an even positive function $\psi\in C^2(\R)$ satisfying $\psi(0)=1$,
	\begin{equation}
		\label{prop1 of psi}
		\lim_{|x|\to\infty}|x|^{-\alpha}\psi(x)=\psi_0
	\end{equation}
	for some constants $\alpha\in\R$ and $\psi_0>0$, and
	\begin{equation}
		\label{prop2 of psi}
		\sup_{x\in\R}\left|\frac{x\psi^\prime(x)}{\psi(x)}\right|<\infty
	\end{equation}
	such that
	\begin{equation}
		\label{eq of psi}
		V(x)=\frac{\psi^{\prime\prime}(x)}{\psi(x)},\quad x\in\R.
	\end{equation}
\end{assumption}

\begin{remark}
	\begin{enumerate}
		\item By the uniqueness of the solution to the Cauchy problem of the linear ordinary differential equations with \eqref{eq of psi} and $(\psi(0),\psi^\prime(0))=(1,0)$, we see that $\psi$ is unique.
		\item Since $\psi$ is positive and \eqref{prop1 of psi} holds, there exist positive constants $\psi_1$ and $\psi_2$ such that $\psi_1\lr{x}^\alpha\le\psi(x)\le\psi_2\lr{x}^\alpha$ for every $x\in\R$.
		\item In the one-dimensional case in \eqref{ass on W}, by the Liouville--Green approximation (see e.g. \cite{O1974}), we see that $W$ satisfies our assumption except for the case where $S$ is subcritical and $\lambda=\lambda_*(1)$ or $S$ is supercritical.
	\end{enumerate}
\end{remark}

We define $p_*(\alpha,m)$ by
\begin{equation}
	\label{def of p_*}
	p_*(\alpha,m)=
	\begin{cases}
		1+\frac{[2-m]_+}{1+\alpha},&\text{if}\ \alpha>\frac{1}{2},\\
		\max\left\{\frac{2}{1+2\alpha},1+\frac{2-m}{1+\alpha}\right\},&\text{if}\ \alpha_*(1)\le\alpha\le\frac{1}{2},\\
		\frac{2}{1+2\alpha},&\text{if}\ -\frac{1}{2}<\alpha<\alpha_*(1),\\
		\infty,&\text{if}\ \alpha\le-\frac{1}{2}.
	\end{cases}
\end{equation}
The main result of the present paper is the following.

\begin{theorem}
	\label{main thm}
	The following statements hold:
	\begin{enumerate}
		\item If $1<p\le p_*(\alpha,m)$, then \eqref{main eq} does not have nontrivial global-in-time solutions;
		\item If $p>p_*(\alpha,m)$, then \eqref{main eq} possesses a nontrivial global-in-time solution for some initial data.
	\end{enumerate}
\end{theorem}

\begin{remark}
	\label{criticality theory}
	In the criticality theory for Schr\"{o}dinger operators, a Schr\"{o}dinger operator in $L^2(\R)$ is said to be subcritical (resp. critical) if it is nonnegative and stable (resp. unstable) under the small perturbation with negative potentials, and a Schr\"{o}dinger operator in $L^2(\R)$ is said to be supercritical if it is not nonnegative.
	In the present paper, the subcriticality corresponds to the case $\alpha>\frac{1}{2}$, and the criticality corresponds to the case $\alpha\le\frac{1}{2}$.
\end{remark}

\begin{remark}
	In Theorem \ref{main thm}, the following two cases provide an answer to the problem which is similar to that left open in \cite{IK2020} ($m=0$):
	\begin{enumerate}
		\item $-\frac{1}{2}<\alpha<\alpha_*(1)$ and $p=p_*(\alpha,0)$;
		\item $\alpha=-\frac{1}{2}$ and $p>1$.
	\end{enumerate}
\end{remark}

Here we would briefly explain the strategy of the present paper.
First, we explain the strategy for the existence part (Theorem \ref{main thm} \textbf{(ii)}).
To this aim, we consider the following three cases:
\begin{center}
\textbf{(I)} $\alpha>\frac{1}{2}$;\quad\textbf{(\II)} $0\le\alpha\le\frac{1}{2}$;\quad\textbf{(\III)} $\alpha<0$.
\end{center}
We shall start by describing the definition of the Schr\"{o}dinger operator $L$.
We define the Schr\"{o}dinger operator $L$ in $L^2(\R)$ by
\[
\begin{cases}
	Lf=-f^{\prime\prime}+Vf,&f\in D(L),\\
	D(L)=H^2(\R).
\end{cases}
\]
Since $L$ is nonnegative and self-adjoint in $L^2(\R)$, we see that $L$ generates an analytic $C_0$-semigroup of contractions $\{e^{-tL}\}_{t\ge0}$ on $L^2(\R)$.
In the cases \textbf{(I)} and \textbf{(\II)}, we deduce the desired decay estimate for $\{e^{-tL}\}_{t\ge0}$ via the Nash-type inequality
\[
\|f\|_{L^2}^{2+\frac{4}{1+2\alpha}}\le C_\textup{N}\|\psi f\|_{L^1}^\frac{4}{1+2\alpha}\|L^\frac{1}{2}f\|_{L^2}^2,\quad f\in H^1(\R)\cap L^{1,\alpha},
\]
where $C_\text{N}$ is a positive constant and
\[
L^{1,\alpha}=\{f\in L^2(\R)\mid\lr{\cdot}^\alpha f\in L^1(\R)\}
\]
(a similar treatment can be found in \cite{S2024}).
Combining the obtained decay estimate for $\{e^{-tL}\}_{t\ge0}$ and some functional inequalities developed in Section \ref{section2}, we can reach the desired result.
In the case \textbf{(\III)}, we deduce an upper bound for $e^{-tL}[\lr{\cdot}^{-1-\alpha}]$ via a suitable supersolution to the corresponding linear equation of \eqref{main eq} with the polynomially decaying property at spatial infinity.
Such a supersolution can be found via the Liouville transform $v\mapsto v^*$ defined by
\[
v^*(H(x),t)=\psi(x)^{-1}v(x,t),\quad H(x)=\int_0^x\psi(y)^{-2}\,dy,\quad x\in\R,\ t\ge0
\]
with the knowledge of \cite{SW2021}.
By the estimate for $e^{-tL}[\lr{\cdot}^{-1-\alpha}]$, we could obtain the global existence.

Next, concerning the nonexistence part (Theorem \ref{main thm} \textbf{(i)}), we divide the proof into the following four cases:
\begin{enumerate}
	\item[\textbf{(A)}] $\alpha\ge\alpha_*(1)$ and $1<p\le1+\frac{[2-m]_+}{1+\alpha}$;
	\item[\textbf{(B)}] $-\frac{1}{2}<\alpha<\frac{1}{2}$ and $1<p<\frac{2}{1+2\alpha}$;
	\item[\textbf{(C)}] $\alpha\le-\frac{1}{2}$ and $p>1$;
	\item[\textbf{(D)}] $-\frac{1}{2}<\alpha<\frac{1}{2}$ and $p=\frac{2}{1+2\alpha}$.
\end{enumerate}
Our treatment is based on the test function method introduced in \cite{MP2001} (TFM).
In the case \textbf{(A)}, the TFM directly provides the desired result.
In the cases \textbf{(B)} and \textbf{(C)}, by the comparison principle, it suffices to prove the global nonexistence for the Cauchy problem of the following equation:
\begin{equation}
	\label{main eq var}
	\pa_tu-\pa_x^2u+Vu=au^p\quad\text{in}\ \R\times(0,\infty),
\end{equation}
where $a\in C_0^\infty(\R)\backslash\{0\}$ is chosen as a nonnegative function with a reasonable property.
Through the harmonic transform $u_*=\psi^{-1}u$, the equation \eqref{main eq var} can be translated into
\begin{equation}
	\label{double parabolic}
	\pa_tu_*-\psi^{-2}\pa_x(\psi^2\pa_xu_*)=a\psi^{p-1}u_*^p\quad\text{in}\ \R\times(0,\infty).
\end{equation}
Then we can expect that the blow-up phenomenon for \eqref{double parabolic} arising from the recurrence of the Brownian motion on the weighted Riemannian manifold $(\R,\psi^2dx)$ occurs.
Since $u\sim u_*$ near the origin, this phenomenon can be seen also for \eqref{main eq var}.
To detect the above phenomenon, we adopt a modified version of the method used in \cite{P1997}.
Together with the TFM, we could arrive at the desired result.
In the case \textbf{(D)}, combining a modified version of the method used in \cite{P1997}, the TFM and the test function method for the corresponding linear equation of \eqref{main eq} introduced in \cite{S2024} (LTFM), we could obtain the desired result.

Finally, we summarize the strategy of the proof of Theorem \ref{main thm} in Table \ref{table1} and Table \ref{table2}.

\begin{table}[htbp]
	\centering
  	\begin{tabular}{|c||c|}\hline
  		Case&Strategy\\ \hline\hline
  		(I)&\multirow{2}{*}{Decay estimate for $\{e^{-tL}\}_{t\ge0}$ \& Some functional inequalities}\\ \cline{1-1}
	  	(\II)&\\ \hline
  		(\III)&Upper bound for $e^{-tL}[\lr{\cdot}^{-1-\alpha}]$\\ \hline
  	\end{tabular}
  	\caption{Strategy of proof of global existence part}
  	\label{table1}
\end{table}

\begin{table}[htbp]
	\centering
  	\begin{tabular}{|c||c|}\hline
  		Case&Strategy\\ \hline\hline
  		(A)&TFM\\ \hline
	  	(B)&\multirow{2}{*}{TFM \& Modified version of method used in \cite{P1997}}\\ \cline{1-1}
  		(C)&\\ \hline
  		(D)&TFM \& Modified version of method used in \cite{P1997} \& LTFM\\ \hline
  	\end{tabular}
  	\caption{Strategy of proof of global nonexistence part}
  	\label{table2}
\end{table}

Throughout the present paper, we use $C$ to denote generic positive constants which may vary in different places, but they are not essential to the analysis.

The rest of the present paper is organized as follows:
In Section \ref{section2}, we collect some fundamental results which are essential in Section \ref{section3} and Section \ref{section4}.
In Section \ref{section3}, we discuss the global existence part (Theorem \ref{main thm} \textbf{(ii)}).
Our proof is based on the so-called blow-up alternative for the cases \textbf{(I)} and \textbf{(\II)}, and the so-called supersolution method for the case \textbf{(\III)}.
In Section \ref{section4}, we treat the nonexistence part (Theorem \ref{main thm} \textbf{(i)}).
Our proof is based on the so-called test function method but not the use of heat kernel estimates.

\section{Preliminaries}
\label{section2}
\subsection{Some functional inequalities and decay estimate for $\{e^{-tL}\}_{t\ge0}$}
In this subsection, we collect some functional inequalities and the decay estimate for $\{e^{-tL}\}_{t\ge0}$ in the case $\alpha\ge0$.
For $q\in[1,\infty]$ and $\sigma\in\R$, we define $L^{q,\sigma}$ by
\[
L^{q,\sigma}=\{f\in L^2(\R)\mid\lr{\cdot}^\sigma f\in L^q(\R)\}.
\]
The following lemma describes the contraction property and the positivity preserving property of $\{e^{-tL}\}_{t\ge0}$ (see \cite{S2024}).

\begin{lemma}
	\label{contr and posi}
	The following statements hold:
	\begin{enumerate}
		\item If $f\in L^{1,\alpha}$, then one has for every $t\ge0$,
		\begin{equation}
			\|\psi(e^{-tL}f)\|_{L^1}\le\|\psi f\|_{L^1};
		\end{equation}
		\item If $f\in L^{\infty,-\alpha}$, then one has for every $t\ge0$,
		\begin{equation}
			\|\psi^{-1}(e^{-tL}f)\|_{L^\infty}\le\|\psi^{-1}f\|_{L^\infty};
		\end{equation}
		\item If $f\in L^2(\R)$ is nonnegative, then $e^{-tL}f$ is also nonnegative for every $t\ge0$.
	\end{enumerate}
\end{lemma}

\begin{proof}
		\textbf{(i)} For $M>0$ and $K>0$, we set
		\begin{gather*}
			\C_M=\{f\in L^2(\R)\mid\|\psi f\|_{L^1}\le M\},\\
			\C_K^*=\{f\in L^2(\R)\mid\|\psi^{-1}f\|_{L^\infty}\le K\}.
		\end{gather*}
		By the Fatou lemma, we see that $\C_M$ is closed in $L^2(\R)$.
		It suffices to prove $e^{-tL}\C_M\subset\C_M$ for every $M>0$.
		First, we prove $(1+\ep L)^{-1}\C_K^*\subset\C_K^*$ for every $\ep>0$ and $K>0$.
		We fix $f\in\C_K^*$ and set $g=(1+\ep L)^{-1}f$, then there exists $\{g_n\}_{n\in\N}\subset C_0^2(\R)$ such that $g_n\to g$ in $H^2(\R)$.
		We set $f_n=(1+\ep L)g_n$, then we have $f_n\to f$ in $L^2(\R)$.
		We fix a nonnegative function $\zeta\in C^1(\R)$ satisfying $0\le\zeta^\prime\in L^\infty(\R)$ and
		\[
		\begin{cases}
			\zeta^\prime(s)>0,&\text{if}\ s>0,\\
			\zeta(s)=0,&\text{if}\ s\le0.
		\end{cases}
		\]
		We set $h_n=\psi^{-1}g_n-K$, then we obtain $\zeta(h_n)\in C_0^1(\R)$.
		Multiplying $\psi\zeta(h_n)$ to the equality $f_n-g_n=-\ep\psi^{-1}(\psi^2h_n^\prime)^\prime$ and integrating it over $\R$, we see from integration by parts that
		\[
		\int_\R(f_n-g_n)\psi\zeta(h_n)\,dx=-\ep\int_\R(\psi^2h_n^\prime)^\prime\zeta(h_n)\,dx=\ep\int_\R\psi^2|h_n^\prime|^2\zeta^\prime(h_n)\,dx\ge0.
		\]
		We set $h=\psi^{-1}g-K$, then we have $\|\psi\zeta(h_n)-\psi\zeta(h)\|_{L^2}\le\|\zeta^\prime\|_{L^\infty}\|g_n-g\|_{L^2}\to 0$ as $n\to\infty$.
		Letting $n\to\infty$ in the above inequality, we obtain
		\[
		\int_\R(f-g)\psi\zeta(h)\,dx\ge0.
		\]
		Therefore we get
		\[
		0\le\int_\R h\psi^2\zeta(h)\,dx=\int_\R(g-K\psi)\psi\zeta(h)\,dx\le\int_\R(f-K\psi)\psi\zeta(h)\,dx\le0.
		\]
		This implies that $g\le K\psi$.
		Similarly, we have $-g\le K\psi$.
		Thus we obtain $g\in C_K^*$.
		
		Next, we prove the following representation for every $M>0$, 
		\[
		\C_M=\left\{f\in L^2(\R)\ \middle|\ \int_\R fg\,dx\le1\quad\text{for every}\ g\in\C_\frac{1}{M}^*\right\}.
		\]
		If $f$ belongs to the right-hand side of the above representation, then we have for every $n\in\N$,
		\[
		\int_\R\frac{|f|^2}{|f|+n^{-1}\psi}\psi\,dx=M\int_\R fg_n\,dx\le M,
		\]
		where $g_n=\frac{n\psi}{M(n|f|+\psi)}f\in\C_\frac{1}{M}^*$.
		Letting $n\to\infty$ in the above inequality, by the monotone convergence theorem, we obtain $f\in\C_M$.
		Conversely, if $f\in\C_M$, then we get for every $g\in\C_{\frac{1}{M}}^*$,
		\[
		\int_\R fg\,dx\le\|\psi f\|_{L^1}\|\psi^{-1}g\|_{L^\infty}\le1.
		\]
		Therefore we see that $f$ belongs to the right-hand side of the above representation.
		
		Finally, we prove $e^{-tL}\C_M\subset\C_M$ for every $M>0$.
		We fix $f\in\C_M$ and $\ep>0$.
		By the self-adjointness of $(1+\ep L)^{-1}$ in $L^2(\R)$ and $(1+\ep L)^{-1}\C_{\frac{1}{M}}^*\subset\C_{\frac{1}{M}}^*$, we have for every $g\in\C_{\frac{1}{M}}^*$,
		\[
		\int_\R((1+\ep L)^{-1}f)g\,dx=\int_\R f((1+\ep L)^{-1}g)\,dx\le\|\psi f\|_{L^1}\|\psi^{-1}((1+\ep L)^{-1}g)\|_{L^\infty}\le1.
		\]
		Combining the above inequality with the above representation, we obtain $(1+\ep L)^{-1}f\in\C_M$.
		This together with the closedness of $\C_M$ in $L^2(\R)$ and the so-called Post--Widder inversion formula (see e.g. \cite{EN2000})
		\[
		\lim_{n\to\infty}\left(1+\frac{t}{n}L\right)^{-n}f=e^{-tL}f\quad\text{in}\ L^2(\R)
		\]
		implies that $e^{-tL}f\in\C_M$.
		The proof is complete.\par\noindent
		\textbf{(ii)} Combining the duality argument, the self-adjointness of $e^{-tL}$ in $L^2(\R)$ and \textbf{(i)}, we have
		\[
		\begin{split}
			\|\psi^{-1}(e^{-tL}f)\|_{L^\infty}&=\sup_{\substack{\varphi\in C_0^\infty(\R)\\ \|\varphi\|_{L^1}\le1}}\int_\R\psi^{-1}(e^{-tL}f)\varphi\,dx\\
			&=\sup_{\substack{\varphi\in C_0^\infty(\R)\\ \|\varphi\|_{L^1}\le1}}\int_\R fe^{-tL}[\psi^{-1}\varphi]\,dx\\
			&\le\|\psi^{-1}f\|_{L^\infty}\sup_{\substack{\varphi\in C_0^\infty(\R)\\ \|\varphi\|_{L^1}\le1}}\|\psi e^{-tL}[\psi^{-1}\varphi]\|_{L^1}\\
			&\le\|\psi^{-1}f\|_{L^\infty}.
		\end{split}
		\]
		The proof is complete.\par\noindent
		\textbf{(iii)} Similarly to the proof of \textbf{(i)}, the proof is complete.\qedhere
\end{proof}

In the case $\alpha>\frac{1}{2}$, the following Hardy-type inequality holds (note that the Hardy inequality does not hold in the one-dimensional case).

\begin{lemma}
	\label{Hardy}
	Assume that $\alpha>\frac{1}{2}$.
	Then the following statements hold:
	\begin{enumerate}
	\item There exists a positive constant $C_{\textup{H},1}$ such that for every $f\in H^1(\R)$,
	\begin{equation}
		\|\lr{\cdot}^{-1}f\|_{L^2}\le C_{\textup{H},1}\|L^\frac{1}{2}f\|_{L^2};
	\end{equation}
	\item There exists a positive constant $C_{\textup{H},2}$ such that for every $f\in H^1(\R)$,
	\begin{equation}
		\|\lr{\cdot}^{-\frac{1}{2}}f\|_{L^\infty}\le C_{\textup{H},2}\|L^\frac{1}{2}f\|_{L^2}.
	\end{equation}
	\end{enumerate}
\end{lemma}

\begin{proof}
	By the density argument, we only need to consider the case $f\in C_0^1(\R)$.
	We set $f_*=\psi^{-1}f\in C_0^1(\R)$, then it is easy to see that $\|\lr{\cdot}^{-1}f\|_{L^2}\le C\|\lr{\cdot}^{\alpha-1}f_*\|_{L^2}$, $\|\lr{\cdot}^{-\frac{1}{2}}f\|_{L^\infty}\le C\|\lr{\cdot}^{\alpha-\frac{1}{2}}f_*\|_{L^\infty}$ and $\|\lr{\cdot}^\alpha f_*^\prime\|_{L^2}\le C\|L^\frac{1}{2}f\|_{L^2}$.
	Therefore it suffices to prove the following two inequalities:
	\begin{gather*}
	\|\lr{\cdot}^{\alpha-1}f_*\|_{L^2}\le C\|\lr{\cdot}^\alpha f_*^\prime\|_{L^2};\\
	\|\lr{\cdot}^{\alpha-\frac{1}{2}}f_*\|_{L^\infty}\le C\|\lr{\cdot}^\alpha f_*^\prime\|_{L^2}.
	\end{gather*}
	\textbf{(i)} It follows from integration by parts and the Schwarz inequality that
	\[
	\begin{split}
		\|\lr{\cdot}^{\alpha-1}f_*\|_{L^2}^2&\le C_{\text{H},1}^\prime\int_\R(x\lr{x}^{2\alpha-2})^\prime|f_*|^2\,dx\\
		&=-2C_{\text{H},1}^\prime\int_\R x\lr{x}^{2\alpha-2}f_*f_*^\prime\,dx\\
		&\le2C_{\text{H},1}^\prime\|\lr{\cdot}^{\alpha-1}f_*\|_{L^2}\|\lr{\cdot}^\alpha f_*^\prime\|_{L^2},
	\end{split}
	\]
	where $C_{\text{H},1}^\prime=\max\left\{1,\frac{1}{2\alpha-1}\right\}$.
	The proof is complete.\par\noindent
	\textbf{(ii)} It is easy to see that for every $x\in\R$,
		\[
		|f_*(x)|^2\le2\min\left\{\int_x^\infty|f_*f_*^\prime|\,dy,\int_{-\infty}^x|f_*f_*^\prime|\,dy\right\}.
		\]
		This together with the Schwarz inequality and \textbf{(i)} implies that for every $x\in\R$,
		\[
		\begin{split}
			\lr{x}^{2\alpha-1}|f_*(x)|^2&\le
			\begin{cases}
				2\int_x^\infty\lr{y}^{2\alpha-1}|f_*f_*^\prime|\,dy,&\text{if}\ x\ge0,\\
				2\int_{-\infty}^x\lr{y}^{2\alpha-1}|f_*f_*^\prime|\,dy,&\text{if}\ x<0
			\end{cases}
			\\
			&\le2\|\lr{\cdot}^{\alpha-1}f_*\|_{L^2}\|\lr{\cdot}^\alpha f_*^\prime\|_{L^2}\\
			&\le C\|\lr{\cdot}^\alpha f_*^\prime\|_{L^2}^2.
		\end{split}
		\]
		The proof is complete.\qedhere
\end{proof}

In the case $\alpha\ge0$, the following weighted Nash-type inequality holds.

\begin{lemma}
	\label{weighted Nash}
	Assume that $\alpha\ge0$.
	Then the following statements hold:
	\begin{enumerate}
		\item There exists a positive constant $C_{\textup{N},1}^*$ such that for every $f\in H^1(\R)\cap L^{1,\alpha}$,
		\begin{equation}
			\|\psi^\frac{2}{3}f\|_{L^2}^6\le C_{\textup{N},1}^*\|\psi f\|_{L^1}^4\|L^\frac{1}{2}f\|_{L^2}^2;
		\end{equation}
		\item There exists a positive constant $C_{\textup{N},2}^*$ such that for every $f\in H^1(\R)\cap L^{1,\alpha}$,
		\begin{equation}
			\|\psi^\frac{1}{3}f\|_{L^\infty}^6\le C_{\textup{N},2}^*\|\psi f\|_{L^1}^2\|L^\frac{1}{2}f\|_{L^2}^4.
		\end{equation}
	\end{enumerate}
\end{lemma}

\begin{proof}
	By the density argument, we only need to consider the case $f\in C_0^1(\R)$.
	We set $f_*=\psi^{-1}f\in C_0^1(\R)$, then it is easy to see that $\|\psi^\frac{2}{3}f\|_{L^2}\le C\|\lr{\cdot}^{\frac{5}{3}\alpha}f_*\|_{L^2}$, $\|\psi^\frac{1}{3}f\|_{L^\infty}\le C\|\lr{\cdot}^{\frac{4}{3}\alpha}f_*\|_{L^\infty}$, $\|\lr{\cdot}^{2\alpha}f_*\|_{L^1}\le C\|\psi f\|_{L^1}$ and $\|\lr{\cdot}^\alpha f_*^\prime\|_{L^2}\le C\|L^\frac{1}{2}f\|_{L^2}$.
	Therefore it suffices to prove the following two inequalities:
	\begin{gather*}
		\|\lr{\cdot}^{\frac{5}{3}\alpha}f_*\|_{L^2}^6\le C\|\lr{\cdot}^{2\alpha}f_*\|_{L^1}^4\|\lr{\cdot}^\alpha f_*^\prime\|_{L^2}^2;\\
		\|\lr{\cdot}^{\frac{4}{3}\alpha}f_*\|_{L^\infty}^6\le C\|\lr{\cdot}^{2\alpha}f_*\|_{L^1}^2\|\lr{\cdot}^\alpha f_*^\prime\|_{L^2}^4.
	\end{gather*}
		\textbf{(i)} It is easy to see that for every $x\in\R$,
		\[
		|f_*(x)|^2\le2\min\left\{\int_x^\infty|f_*f_*^\prime|\,dy,\int_{-\infty}^x|f_*f_*^\prime|\,dy\right\}.
		\]
		This together with the Schwarz inequality implies that for every $x\in\R$,
		\[
		\begin{split}
			\lr{x}^{\frac{8}{3}\alpha}|f_*(x)|^2&\le
			\begin{cases}
				2\int_x^\infty\lr{y}^{\frac{8}{3}\alpha}|f_*f_*^\prime|\,dy,&\text{if}\ x\ge0,\\
				2\int_{-\infty}^x\lr{y}^{\frac{8}{3}\alpha}|f_*f_*^\prime|\,dy,&\text{if}\ x<0
			\end{cases}
			\\
			&\le2\|\lr{\cdot}^{\frac{5}{3}\alpha}f_*\|_{L^2}\|\lr{\cdot}^\alpha f_*^\prime\|_{L^2}.
		\end{split}
		\]
		Therefore we have
		\[
		\|\lr{\cdot}^{\frac{5}{3}\alpha}f_*\|_{L^2}^8\le\|\lr{\cdot}^{\frac{4}{3}\alpha}f_*\|_{L^\infty}^4\|\lr{\cdot}^{2\alpha}f_*\|_{L^1}^4\le4\|\lr{\cdot}^{\frac{5}{3}\alpha}f_*\|_{L^2}^2\|\lr{\cdot}^\alpha f_*^\prime\|_{L^2}^2\|\lr{\cdot}^{2\alpha}f_*\|_{L^1}^4.
		\]
		The proof is complete.\par\noindent
		\textbf{(ii)} Similarly to the proof of \textbf{(i)}, we have
		\[
		\|\lr{\cdot}^{\frac{4}{3}\alpha}f_*\|_{L^\infty}^2\le2\|\lr{\cdot}^{\frac{5}{3}\alpha}f_*\|_{L^2}\|\lr{\cdot}^\alpha f_*^\prime\|_{L^2}\le C\|\lr{\cdot}^{2\alpha}f_*\|_{L^1}^\frac{2}{3}\|\lr{\cdot}^\alpha f_*^\prime\|_{L^2}^\frac{4}{3}.
		\]
		The proof is complete.\qedhere
\end{proof}

Combining Lemma \ref{Hardy} \textbf{(i)} and Lemma \ref{weighted Nash}, we deduce the following Nash-type inequality.

\begin{lemma}
	\label{Nash}
	Assume that $\alpha\ge0$.
	Then there exists a positive constant $C_\textup{N}$ such that for every $f\in H^1(\R)\cap L^{1,\alpha}$,
	\begin{equation}
		\|f\|_{L^2}^{2+\frac{4}{1+2\alpha}}\le C_\textup{N}\|\psi f\|_{L^1}^\frac{4}{1+2\alpha}\|L^\frac{1}{2}f\|_{L^2}^2.
	\end{equation}
\end{lemma}

\begin{proof}
	First, we consider the case $\alpha>\frac{1}{2}$.
	Combining Lemma \ref{Hardy} \textbf{(i)} and Lemma \ref{weighted Nash} \textbf{(i)}, we have
	\[
	\begin{split}
		\|f\|_{L^2}&\le C\|\lr{\cdot}^{-1}f\|_{L^2}^\frac{2\alpha}{3+2\alpha}\|\psi^\frac{2}{3}f\|_{L^2}^\frac{3}{3+2\alpha}\\
		&\le C\|\psi f\|_{L^1}^\frac{2}{3+2\alpha}\|L^\frac{1}{2}f\|_{L^2}^\frac{1+2\alpha}{3+2\alpha}.
	\end{split}
	\]
	
	Finally, we consider the case $0\le\alpha\le\frac{1}{2}$.
	We may assume that $\|\psi f\|_{L^1},\|L^\frac{1}{2}f\|_{L^2}\neq0$.
	Combining Lemma \ref{weighted Nash} \textbf{(i)} and Lemma \ref{weighted Nash} \textbf{(ii)}, we have for every $R>0$,
	\[
	\begin{split}
		\|f\|_{L^2}^2&=\int_{|x|\le R}\psi^{-\frac{2}{3}}\psi^\frac{2}{3}|f|^2\,dx+\int_{|x|\ge R}\psi^{-\frac{4}{3}}\psi^\frac{4}{3}|f|^2\,dx\\
		&\le CR^{1-\frac{2}{3}\alpha}\|\psi f\|_{L^1}^\frac{2}{3}\|L^\frac{1}{2}f\|_{L^2}^\frac{4}{3}+CR^{-\frac{4}{3}\alpha}\|\psi f\|_{L^1}^\frac{4}{3}\|L^\frac{1}{2}f\|_{L^2}^\frac{2}{3}.
	\end{split}
	\]
	Substituting $R=\left(\frac{\|\psi f\|_{L^1}}{\|L^\frac{1}{2}f\|_{L^2}}\right)^{\frac{2}{3+2\alpha}}$ into the above inequality, we have the desired result.
\end{proof}

Combining Lemma \ref{contr and posi} \textbf{(i)} and Lemma \ref{Nash}, we deduce the following decay estimate for $\{e^{-tL}\}_{t\ge0}$ (see e.g. \cite{O2005}).

\begin{lemma}
	\label{decay}
	Assume that $\alpha\ge0$.
	Then the following statements hold:
	\begin{enumerate}
		\item For every $1\le q_1\le q_2\le\infty$, there exists a positive constant $C_{q_1,q_2}$ such that if $f\in L^{q_1,\alpha\left(\frac{2}{q_1}-1\right)}$, then one has for every $t>0$,
		\begin{equation}
			\|\psi^{\frac{2}{q_2}-1}(e^{-tL}f)\|_{L^{q_2}}\le C_{q_1,q_2}t^{-\frac{1+2\alpha}{2}\left(\frac{1}{q_1}-\frac{1}{q_2}\right)}\|\psi^{\frac{2}{q_1}-1}f\|_{L^{q_1}};
		\end{equation}
		\item For every $q\in[1,2]$, there exists a positive constant $C_q$ such that if $f\in L^{q,\alpha\left(\frac{2}{q}-1\right)}$, then one has for every $t>0$,
		\begin{equation}
			\|L^\frac{1}{2}e^{-tL}f\|_{L^2}\le C_qt^{-\frac{1+2\alpha}{2}\left(\frac{1}{q}-\frac{1}{2}\right)-\frac{1}{2}}\|\psi^{\frac{2}{q}-1}f\|_{L^q}.
		\end{equation}
	\end{enumerate}
\end{lemma}

\subsection{Blow-up alternative and supersolution method}
In this subsection, we state the so-called blow-up alternative and the so-called supersolution method.
To this aim, we also use the symbol $\{e^{-tL}\}_{t\ge0}$ to denote the semigroup $\{T(t)\}_{t\ge0}$ on $BC(\R)$ with the quasi-contraction property defined through the heat kernel of $L$ (note that $\{T(t)\}_{t\ge0}$ coincides with $\{e^{-tL}\}_{t\ge0}$ on $L^2(\R)\cap BC(\R)$).
The following lemma describes the existence and the uniqueness of the local-in-time solution to \eqref{main eq} and the blow-up alternative (see e.g. \cite{QS2007}). 

\begin{lemma}
	\label{blow-up alt}
	There exist $T=T_\textup{max}(u_0)\in(0,\infty]$ and a unique local-in-time solution $u\in C^{2;1}(\R\times(0,T))\cap C(\R\times[0,T))$ to \eqref{main eq} with the following properties:
	\begin{enumerate}
		\item For every $t\in[0,T)$,
		\begin{equation}
			\label{bounded sol}
			u(t)\in BC(\R)
		\end{equation}
		and
		\begin{equation}
			\label{int eq}
			u(t)=e^{-tL}u_0+\int_0^te^{-(t-s)L}[\lr{\cdot}^{-m}u(s)^p]\,ds\quad\text{in}\ \R;
		\end{equation}
		\item If $T<\infty$, then $u$ satisfies
		\begin{equation}
			\label{blow-up}
			\lim_{t\uparrow T}\|u(t)\|_{L^\infty}=\infty.
		\end{equation}
	\end{enumerate}
\end{lemma}

The following lemma describes the supersolution method (see e.g. \cite{IK2020}).

\begin{lemma}
	\label{supersol method}
	Assume that there exists a positive function $\supu\in BC(\R\times(0,\infty))$ such that for every $t>0$,
	\[
	\supu(t)\ge e^{-tL}u_0+\int_0^te^{-(t-s)L}[\lr{\cdot}^{-m}\supu(s)^p]\,ds\quad\text{in}\ \R.
	\]
	Then \eqref{main eq} possesses a global-in-time solution.
\end{lemma}

\section{Existence of nontrivial global-in-time solutions}
\label{section3}
In this section, we prove Theorem \ref{main thm} \textbf{(ii)}.
To this aim, we consider the following three cases (cf. Table \ref{table1}):
\begin{center}
\textbf{(I)} $\alpha>\frac{1}{2}$;\quad\textbf{(\II)} $0\le\alpha\le\frac{1}{2}$;\quad\textbf{(\III)} $\alpha<0$.
\end{center}
First, we prove the case \textbf{(I)}.

\begin{lemma}
	\label{case I}
	In the case \textbf{\textup{(I)}}, \eqref{main eq} possesses a nontrivial global-in-time solution for some initial data.
\end{lemma}

\begin{proof}
	By the comparison principle, we only need to consider the case $m\le2$.
	We fix a nonnegative function $u_0\in C_0^\infty(\R)\backslash\{0\}$ satisfying $\|\psi u_0\|_{L^1}+\|L^{\frac{1}{2}}u_0\|_{L^2}\le\ep$ for sufficiently small $\ep>0$. 
	By Lemma \ref{blow-up alt}, there exist $T\in(0,\infty]$ and a local-in-time solution $u\in C^{2;1}(\R\times(0,T))\cap C(\R\times[0,T))$ to \eqref{main eq} satisfying \eqref{bounded sol}, \eqref{int eq} and \eqref{blow-up} if $T<\infty$.
	We set
	\[
	M(T)=\sup_{0\le t<T}(\|\psi u(t)\|_{L^1}+(1+t)^{\frac{3+2\alpha}{4}}\|L^{\frac{1}{2}}u(t)\|_{L^2}).
	\]
	We fix $t\in(0,T)$.
	First, we prove the following inequality:
	\begin{equation}
	\label{ineq I1}
	\|\psi u(t)\|_{L^1}\le\ep+CM(T)^p.
	\end{equation}
	By Lemma \ref{Hardy} \textbf{(ii)}, we have for every $s\in[0,t]$,
	\begin{equation}
		\label{ineq I2}
		\|\lr{\cdot}^{-\frac{1}{2}}u(s)\|_{L^\infty}\le C\|L^\frac{1}{2}u(s)\|_{L^2}\le CM(T)(1+s)^{-\frac{3+2\alpha}{4}}.
	\end{equation}
	This together with Lemma \ref{weighted Nash} \textbf{(ii)} implies that for every $s\in[0,t]$,
	\begin{equation}
		\label{ineq I3}
		\begin{split}
		\|u(s)\|_{L^\infty}&\le C\|\lr{\cdot}^{-\frac{1}{2}}u(s)\|_{L^\infty}^\frac{2\alpha}{3+2\alpha}\|\psi^\frac{1}{3}u(s)\|_{L^\infty}^\frac{3}{3+2\alpha}\\
		&\le C\|\lr{\cdot}^{-\frac{1}{2}}u(s)\|_{L^\infty}^\frac{2\alpha}{3+2\alpha}\|\psi u(s)\|_{L^1}^\frac{1}{3+2\alpha}\|L^\frac{1}{2}u(s)\|_{L^2}^\frac{2}{3+2\alpha}\\
		&\le CM(T)(1+s)^{-\frac{1+\alpha}{2}}.
		\end{split}
	\end{equation}
	It follows from \eqref{ineq I2} that for every $s\in[0,t]$,
	\[
	\|\lr{\cdot}^{-\frac{3}{2}-\delta_1}u(s)\|_{L^1}\le\|\lr{\cdot}^{-\frac{1}{2}}u(s)\|_{L^\infty}\|\lr{\cdot}^{-1-\delta_1}\|_{L^1}\le CM(T)(1+s)^{-\frac{3+2\alpha}{4}},
	\]
	where $\delta_1>0$ is a sufficiently small constant.
	Combining the above inequality with \eqref{ineq I3}, we obtain for every $s\in[0,t]$,
	\begin{equation}
	\label{ineq I5}
	\begin{split}
	\|\psi\lr{\cdot}^{-m}u(s)^p\|_{L^1}&\le C\|u(s)\|_{L^\infty}^{p-1}\|\psi u(s)\|_{L^1}^{1-\frac{2m}{3+2\alpha+2\delta_1}}\|\lr{\cdot}^{-\frac{3}{2}-\delta_1}u(s)\|^{\frac{2m}{3+2\alpha+2\delta_1}}\\
	&\le CM(T)^p(1+s)^{-\frac{1+\alpha}{2}(p-1)-\frac{3+2\alpha}{4}\frac{2m}{3+2\alpha+2\delta_1}}\\
	&\le CM(T)^p(1+s)^{-1-\delta_2},
	\end{split}
	\end{equation}
	where $\delta_2>0$ is a sufficiently small constant.
	Combining \eqref{int eq}, Lemma \ref{contr and posi} \textbf{(i)} and \eqref{ineq I5}, we get
	\[
		\begin{split}
			\|\psi u(t)\|_{L^1}&\le\|\psi(e^{-tL}u_0)\|_{L^1}+\int_0^t\|\psi e^{-(t-s)L}[\lr{\cdot}^{-m}u(s)^p]\|_{L^1}\,ds\\
			&\le\|\psi u_0\|_{L^1}+\int_0^t\|\psi\lr{\cdot}^{-m}u(s)^p\|_{L^1}\,ds\\
			&\le\ep+CM(T)^p\int_0^t(1+s)^{-1-\delta_2}\,ds\\
			&\le\ep+CM(T)^p.
		\end{split}
	\]
	
	Next, we prove the following inequality:
	\begin{equation}
	\label{ineq I6}
	(1+t)^{\frac{3+2\alpha}{4}}\|L^\frac{1}{2}u(t)\|_{L^2}\le C\ep+CM(T)^p.
	\end{equation}
	Noting that $\|L^{\frac{1}{2}}\lr{\cdot}^{-m}u(s)^p\|_{L^2}\le CM(T)^p$ for every $s\in[0,t]$, we see from Lemma \ref{decay} \textbf{(ii)} and \eqref{ineq I5} that
	\begin{equation}
	\label{ineq I7}
	\begin{split}
	&\int_0^\frac{t}{2}\|L^{\frac{1}{2}}e^{-(t-s)L}[\lr{\cdot}^{-m}u(s)^p]\|_{L^2}\,ds\\
	&\le C\int_0^\frac{t}{2}((1+t-s)^{-\frac{3+2\alpha}{4}}\|\psi\lr{\cdot}^{-m}u(s)^p\|_{L^1}+e^{-(t-s)}\|L^{\frac{1}{2}}\lr{\cdot}^{-m}u(s)^p\|_{L^2})\,ds\\
	&\le CM(T)^p\int_0^\frac{t}{2}(1+t-s)^{-\frac{3+2\alpha}{4}}(1+s)^{-1-\delta_2}\,ds+CM(T)^pte^{-\frac{t}{2}}\\
	&\le CM(T)^p(1+t)^{-\frac{3+2\alpha}{4}}.
	\end{split}
	\end{equation}
	We fix $q\in[1,2]$ satisfying $\frac{1}{q}=\frac{1}{2}+\frac{1-2\delta_2}{1+2\alpha}$.
	In the case $m\le1$, noting that $\frac{1+2\alpha}{4}\left(\frac{2}{q}-1\right)=\frac{1}{2}-\delta_2$ and combining Lemma \ref{Hardy} \textbf{(i)} and Lemma \ref{Nash}, we have
	\begin{equation}
	\label{ineq I8}
	\begin{split}
	&\|\psi^{\frac{2}{q}-1}\lr{\cdot}^{-m}u(s)^p\|_{L^q}\\
	&\le\|u(s)\|_{L^\infty}^{p-1}\|\lr{\cdot}^{-m}u(s)\|_{L^2}^{2-\frac{2}{q}}\|\psi\lr{\cdot}^{-m}u(s)\|_{L^1}^{\frac{2}{q}-1}\\
	&\le\|u(s)\|_{L^\infty}^{p-1}\|\lr{\cdot}^{-1}u(s)\|_{L^2}^{m\left(2-\frac{2}{q}\right)}\|u(s)\|_{L^2}^{(1-m)\left(2-\frac{2}{q}\right)}\|\psi\lr{\cdot}^{-m}u(s)\|_{L^1}^{\frac{2}{q}-1}\\
	&\le\|u(s)\|_{L^\infty}^{p-1}\|L^\frac{1}{2}u(s)\|_{L^2}^{\frac{1+2\alpha+2m}{3+2\alpha}\left(2-\frac{2}{q}\right)}\|\psi u(s)\|_{L^1}^{\frac{2}{3+2\alpha}(1-m)\left(2-\frac{2}{q}\right)}\|\psi\lr{\cdot}^{-m}u(s)\|_{L^1}^{\frac{2}{q}-1}\\
	&\le CM(T)^p(1+s)^{-\frac{1+\alpha}{2}(p-1)-\frac{m}{2}-\frac{1+2\alpha}{4}+\frac{1+2\alpha}{4}\left(\frac{2}{q}-1\right)+\delta_3}\\
	&\le CM(T)^p(1+s)^{-\frac{3+2\alpha}{4}-\delta_2},
	\end{split}
	\end{equation}
	where $\delta_3>0$ is a sufficiently small constant.
	Similarly, in the case $m>1$, noting that $\alpha\left(\frac{2}{q}-1\right)=\frac{3}{2}-\frac{1}{q}-2\delta_2$, we obtain
	\[
	\begin{split}
	\|\psi^{\frac{2}{q}-1}\lr{\cdot}^{-m}u(s)^p\|_{L^q}&\le C\|\lr{\cdot}^{\frac{3}{2}-\frac{1}{q}-2\delta_2-m}u(s)^p\|_{L^q}\\
	&\le C\|u(s)\|_{L^\infty}^{p-1}\|\lr{\cdot}^{1-m}u(s)\|_{L^2}^{2-\frac{2}{q}}\|\lr{\cdot}^{\frac{1}{2}-m-\delta_4}u(s)\|_{L^1}^{\frac{2}{q}-1}\\
	&\le CM(T)^p(1+s)^{-\frac{1+\alpha}{2}(p-1)-\frac{m}{2}+\frac{1}{4}-\frac{\alpha}{2}+\delta_5}\\
	&\le CM(T)^p(1+s)^{-\frac{3+2\alpha}{4}-\delta_2},
	\end{split}
	\]
	where $\delta_4,\delta_5>0$ are sufficiently small constants.
	Combining the above inequality with Lemma \ref{decay} \textbf{(ii)} and \eqref{ineq I8}, we get
	\[
	\begin{split}
	\int_\frac{t}{2}^t\|L^{\frac{1}{2}}e^{-(t-s)L}[\lr{\cdot}^{-m}u(s)^p]\|_{L^2}\,ds&\le C\int_\frac{t}{2}^t(t-s)^{-1+\delta_2}\|\psi^{\frac{2}{q}-1}\lr{\cdot}^{-m}u(s)^p\|_{L^q}\,ds\\
	&\le CM(T)^p\int_\frac{t}{2}^t(t-s)^{-1+\delta_2}(1+s)^{-\frac{3+2\alpha}{4}-\delta_2}\,ds\\
	&\le CM(T)^p(1+t)^{-\frac{3+2\alpha}{4}}.
	\end{split}
	\]
	This together with \eqref{int eq}, Lemma \ref{decay} \textbf{(ii)} and \eqref{ineq I7}, we see that
	\[
	\begin{split}
	\|L^{\frac{1}{2}}u(t)\|_{L^2}&\le\|L^{\frac{1}{2}}e^{-tL}u_0\|_{L^2}+\int_0^t\|L^{\frac{1}{2}}e^{-(t-s)L}[\lr{\cdot}^{-m}u(s)^p]\|_{L^2}\,ds\\
	&\le C(1+t)^{-\frac{3+2\alpha}{4}}(\|\psi u_0\|_{L^1}+\|L^{\frac{1}{2}}u_0\|_{L^2})+CM(T)^p(1+t)^{-\frac{3+2\alpha}{4}}\\
	&\le C(\ep+M(T)^p)(1+t)^{-\frac{3+2\alpha}{4}}.
	\end{split}
	\]
	
	Finally, we prove $T=\infty$.
	We assume that $T<\infty$.
	Combining \eqref{ineq I1} and \eqref{ineq I6}, we have $M(T)\le C\ep+CM(T)^p$.
	Since $\ep$ is sufficiently small, we see that $M(T)\le C\ep$.
	Therefore, by \eqref{ineq I3}, we obtain
	\[
	\|u(t)\|_{L^\infty}\le CM(T)(1+t)^{-\frac{1+\alpha}{2}}\le C\ep(1+t)^{-\frac{1+\alpha}{2}},
	\]
	and this contradicts \eqref{blow-up}.
	The proof is complete.
\end{proof}

Next, we prove the case \textbf{(\II)}.
In this case, since the Hardy-type inequality does not hold, we use the $L^{\infty,-\alpha}$-$L^{1,\alpha}$ decay estimate for $\{e^{-tL}\}_{t\ge0}$ instead of the Hardy-type one.

\begin{lemma}
	\label{case II}
	In the case \textbf{\textup{(\II)}}, \eqref{main eq} possesses a nontrivial global-in-time solution for some initial data.
\end{lemma}

\begin{proof}
	By the comparison principle, we only need to consider the case $m\le1+\alpha+\frac{2\alpha}{1+2\alpha}$.
	We fix $\delta>0$ satisfying $\frac{1+\alpha}{2}(p-1)+\frac{m}{2}-1>\delta$ and $\frac{1+2\alpha}{2}p-1>\delta$, and a nonnegative function $u_0\in C_0^\infty(\R)\backslash\{0\}$ satisfying $\|\psi u_0\|_{L^1}+\|\psi^{-1}u_0\|_{L^\infty}+\|L^{\frac{1}{2}}u_0\|_{L^2}\le\ep$ for sufficiently small $\ep>0$. 
	By Lemma \ref{blow-up alt}, there exist $T\in(0,\infty]$ and a local-in-time solution $u\in C^{2;1}(\R\times(0,T))\cap C(\R\times[0,T))$ to \eqref{main eq} satisfying \eqref{bounded sol}, \eqref{int eq} and \eqref{blow-up} if $T<\infty$.
	We set
	\[
	M(T)=\sup_{0\le t<T}(\|\psi u(t)\|_{L^1}+(1+t)^{\frac{1+2\alpha}{2}}\|\psi^{-1}u(t)\|_{L^\infty}+(1+t)^{\frac{3+2\alpha}{4}}\|L^{\frac{1}{2}}u(t)\|_{L^2}).
	\]
	We fix $t\in(0,T)$.
	First, we prove the following inequality:
	\begin{equation}
	\label{ineq II1}
	\|\psi u(t)\|_{L^1}\le\ep+CM(T)^p.
	\end{equation}
	By Lemma \ref{weighted Nash} \textbf{(ii)}, we have for every $s\in[0,t]$,
	\begin{equation}
	\label{ineq II1-5}
	\begin{split}
		\|u(s)\|_{L^\infty}&\le\|\psi^{-1}u(s)\|_{L^\infty}^{\frac{1}{4}}\|\psi^{\frac{1}{3}}u(s)\|_{L^\infty}^{\frac{3}{4}}\\
		&\le C\|\psi^{-1}u(s)\|_{L^\infty}^{\frac{1}{4}}\|\psi u(s)\|_{L^1}^{\frac{1}{4}}\|L^{\frac{1}{2}}u(s)\|_{L^2}^{\frac{1}{2}}\\
		&\le CM(T)(1+s)^{-\frac{1+\alpha}{2}}.
	\end{split}
	\end{equation}
	This implies that for every $s\in[0,t]$,
	\begin{equation}
		\label{ineq II2}
		\begin{split}
			\|\lr{\cdot}^{-m}u(s)^{p-1}\|_{L^\infty}&\le
			\begin{cases}
				\|\psi^{-1}u(s)\|_{L^\infty}^{\frac{m}{\alpha}}\|u(s)\|_{L^\infty}^{p-1-\frac{m}{\alpha}},&\text{if}\ m<\alpha(p-1),\\
				\|\psi^{-1}u(s)\|_{L^\infty}^{p-1},&\text{if}\ m\ge\alpha(p-1)
			\end{cases}
			\\
			&\le CM(T)^{p-1}
			\begin{cases}
				(1+s)^{-\frac{1+\alpha}{2}(p-1)-\frac{m}{2}},&\text{if}\ m<\alpha(p-1),\\
				(1+s)^{-\frac{1+2\alpha}{2}(p-1)},&\text{if}\ m\ge\alpha(p-1).
			\end{cases}
		\end{split}
	\end{equation}
	It follows that for every $s\in[0,t]$,
	\begin{equation}
		\label{ineq II3}
		\|\psi^{-1}u(s)\|_{L^1}=\int_{|x|\le\sqrt{1+s}}\psi^{-1}u(s)\,dx+\int_{|x|\ge\sqrt{1+s}}\psi^{-2}\psi u(s)\,dx\le CM(T)(1+s)^{-\alpha}.
	\end{equation}
	This implies that for every $q\in[1,\infty]$ and $s\in[0,t]$,
	\begin{equation}
		\label{ineq II4}
		\|\psi^{-p}u(s)^p\|_{L^q}\le\|\psi^{-1}u(s)\|_{L^\infty}^{p-\frac{1}{q}}\|\psi^{-1}u(s)\|_{L^1}^{\frac{1}{q}}\le CM(T)^p(1+s)^{-\frac{1+2\alpha}{2}p+\frac{1}{2q}}.
	\end{equation}
	We fix $\rho>0$ satisfying $\frac{1+\alpha}{2}(p-1)+\frac{m}{2(1+\rho)}-1>\delta$.
	In the case $m>\alpha(p+1)$, we set $q_1=\frac{1+\rho}{m-\alpha(p+1)}>1$ and $q_2=\frac{q_1}{q_1-1}$.
	Combining \eqref{ineq II2}, \eqref{ineq II3} and \eqref{ineq II4}, we obtain for every $s\in[0,t]$,
	\begin{equation}
		\label{ineq II5}
		\begin{split}
			&\|\psi\lr{\cdot}^{-m}u(s)^p\|_{L^1}\\
			&\le
			\begin{cases}
				\|\lr{\cdot}^{-m}u(s)^{p-1}\|_{L^\infty}\|\psi u(s)\|_{L^1},&\text{if}\ m<\alpha(p-1),\\
				\|\psi^{-1}u(s)\|_{L^\infty}^{p-1}\|\psi^{-1}u(s)\|_{L^1}^{\frac{m-\alpha(p-1)}{2\alpha}}\|\psi u(s)\|_{L^1}^{\frac{\alpha(p+1)-m}{2\alpha}},&\text{if}\ \alpha(p-1)\le m\le\alpha(p+1),\\
				\|\lr{\cdot}^{\alpha(p+1)-m}\|_{L^{q_1}}\|\psi^{-p}u(s)^p\|_{L^{q_2}},&\text{if}\ m>\alpha(p+1)
			\end{cases}
			\\
			&\le CM(T)^p
			\begin{cases}
				(1+s)^{-\frac{1+\alpha}{2}(p-1)-\frac{m}{2}},&\text{if}\ m<\alpha(p-1),\\
				(1+s)^{-\frac{1+2\alpha}{2}(p-1)-\frac{1}{2}(m-\alpha(p-1))},&\text{if}\ \alpha(p-1)\le m\le\alpha(p+1),\\
				(1+s)^{-\frac{1+\alpha}{2}(p-1)-\frac{m}{2(1+\rho)}},&\text{if}\ m>\alpha(p+1)
			\end{cases}
			\\
			&\le CM(T)^p(1+s)^{-1-\delta}.
		\end{split}
	\end{equation}
	This together with \eqref{int eq} and Lemma \ref{contr and posi} \textbf{(i)} implies that
	\[
		\begin{split}
			\|\psi u(t)\|_{L^1}&\le\|\psi(e^{-tL}u_0)\|_{L^1}+\int_0^t\|\psi e^{-(t-s)L}[\lr{\cdot}^{-m}u(s)^p]\|_{L^1}\,ds\\
			&\le\|\psi u_0\|_{L^1}+\int_0^t\|\psi\lr{\cdot}^{-m}u(s)^p\|_{L^1}\,ds\\
			&\le\ep+CM(T)^p\int_0^t(1+s)^{-1-\delta}\,ds\\
			&\le\ep+CM(T)^p.
		\end{split}
	\]
	
	Next, we prove the following inequality:
	\begin{equation}
	\label{ineq II6}
	(1+t)^\frac{1+2\alpha}{2}\|\psi^{-1}u(t)\|_{L^\infty}\le C\ep+CM(T)^p.
	\end{equation}
	Combining Lemma \ref{decay} \textbf{(i)}, Lemma \ref{contr and posi} \textbf{(ii)}, \eqref{ineq II5} and \eqref{ineq II2}, we have
	\[
	\begin{split}
	&\int_0^t\|\psi^{-1}e^{-(t-s)L}[\lr{\cdot}^{-m}u(s)^p]\|_{L^\infty}\,ds\\
	&\le C\int_0^t(1+t-s)^{-\frac{1+2\alpha}{2}}(\|\psi\lr{\cdot}^{-m}u(s)^p\|_{L^1}+\|\psi^{-1}\lr{\cdot}^{-m}u(s)^p\|_{L^\infty})\,ds\\
	&\le CM(T)^p\int_0^t(1+t-s)^{-\frac{1+2\alpha}{2}}(1+s)^{-1-\delta}\,ds\\
	&\le CM(T)^p(1+t)^{-\frac{1+2\alpha}{2}}.
	\end{split}
	\]
	This together with \eqref{int eq}, Lemma \ref{decay} \textbf{(i)} and Lemma \ref{contr and posi} \textbf{(ii)} implies that
	\[
		\begin{split}
			\|\psi^{-1}u(t)\|_{L^\infty}&\le\|\psi^{-1}(e^{-tL}u_0)\|_{L^\infty}+\int_0^t\|\psi^{-1}e^{-(t-s)L}[\lr{\cdot}^{-m}u(s)^p]\|_{L^\infty}\,ds\\
			&\le C(1+t)^{-\frac{1+2\alpha}{2}}(\|\psi u_0\|_{L^1}+\|\psi^{-1}u_0\|_{L^\infty})+CM(T)^p(1+t)^{-\frac{1+2\alpha}{2}}\\
			&\le C(\ep+M(T)^p)(1+t)^{-\frac{1+2\alpha}{2}}.
		\end{split}
	\]
	
	Next, we prove the following inequality:
	\begin{equation}
	\label{ineq II9}
	(1+t)^{\frac{3+2\alpha}{4}}\|L^\frac{1}{2}u(t)\|_{L^2}\le C\ep+CM(T)^p.
	\end{equation}
	Noting that $\|L^\frac{1}{2}\lr{\cdot}^{-m}u(s)^p\|_{L^2}\le CM(T)^p$ for every $s\in[0,t]$ and combining Lemma \ref{decay} \textbf{(ii)} and \eqref{ineq II5}, we have
	\[
	\begin{split}
	&\int_0^\frac{t}{2}\|L^\frac{1}{2}e^{-(t-s)L}[\lr{\cdot}^{-m}u(s)^p]\|_{L^2}\,ds\\
	&\le C\int_0^\frac{t}{2}((1+t-s)^{-\frac{3+2\alpha}{4}}\|\psi\lr{\cdot}^{-m}u(s)^p\|_{L^1}+e^{-(t-s)}\|L^\frac{1}{2}\lr{\cdot}^{-m}u(s)^p\|_{L^2})\,ds\\
	&\le CM(T)^p\int_0^\frac{t}{2}(1+t-s)^{-\frac{3+2\alpha}{4}}(1+s)^{-1-\delta}\,ds+CM(T)^pte^{-\frac{t}{2}}\\
	&\le CM(T)^p(1+t)^{-\frac{3+2\alpha}{4}}.
	\end{split}
	\]
	By Lemma \ref{decay} \textbf{(ii)} and \eqref{ineq II5}, we obtain
	\[
	\begin{split}
	\int_\frac{t}{2}^t\|L^\frac{1}{2}e^{-(t-s)L}[\lr{\cdot}^{-m}u(s)^p]\|_{L^2}\,ds&\le C\int_\frac{t}{2}^t(t-s)^{-\frac{3+2\alpha}{4}}\|\psi\lr{\cdot}^{-m}u(s)^p\|_{L^1}\,ds\\
	&\le CM(T)^p\int_\frac{t}{2}^t(t-s)^{-\frac{3+2\alpha}{4}}(1+s)^{-1-\delta}\,ds\\
	&\le CM(T)^p(1+t)^{-1-\delta}t^{1-\frac{3+2\alpha}{4}}\log(1+t)\\
	&\le CM(T)^p(1+t)^{-\frac{3+2\alpha}{4}}.
	\end{split}
	\]
	Therefore, combining \eqref{int eq} and Lemma \ref{decay} \textbf{(ii)}, we get
	\[
		\begin{split}
			\|L^\frac{1}{2}u(t)\|_{L^2}&\le\|L^\frac{1}{2}(e^{-tL}u_0)\|_{L^2}+\int_0^t\|L^\frac{1}{2}e^{-(t-s)L}[\lr{\cdot}^{-m}u(s)^p]\|_{L^2}\,ds\\
			&\le C(1+t)^{-\frac{3+2\alpha}{4}}(\|\psi u_0\|_{L^1}+\|L^\frac{1}{2}u_0\|_{L^2})+CM(T)^p(1+t)^{-\frac{3+2\alpha}{4}}\\
			&\le C(\ep+M(T)^p)(1+t)^{-\frac{3+2\alpha}{4}}.
		\end{split}
	\]
	
	Finally, we prove $T=\infty$.
	We assume that $T<\infty$.
	Combining \eqref{ineq II1}, \eqref{ineq II6} and \eqref{ineq II9}, we have $M(T)\le C\ep+CM(T)^p$.
	Since $\ep$ is sufficiently small, we see that $M(T)\le C\ep$.
	Therefore, by \eqref{ineq II1-5}, we obtain
	\[
	\|u(t)\|_{L^\infty}\le CM(T)(1+t)^{-\frac{1+\alpha}{2}}\le C\ep(1+t)^{-\frac{1+\alpha}{2}},
	\]
	and this contradicts \eqref{blow-up}.
	The proof is complete.
\end{proof}

Next, we prove the case \textbf{(\III)}.
In this case, our argument relies on the comparison principle.
To this aim, we deduce an upper bound for $e^{-tL}[\lr{\cdot}^{-1-\alpha}]$ via the following supersolution to a certain parabolic equation constructed in \cite{SW2021}.

\begin{lemma}
	\label{supersol of Sobajima--Wakasugi}
	Assume that a positive function $D\in C(\R)$ satisfies
	\begin{equation}
		\label{prop of D}
		\lim_{|y|\to\infty}|y|^{-\beta}D(y)=D_0
	\end{equation}
	for some constants $\beta<1$ and $D_0>0$, and $\rho\in\left(0,\frac{1-\beta}{2}\right)$.
	Then there exist a supersolution $\overline{v^*}\in C^2(\R\times[0,\infty))$ to the Cauchy problem of the parabolic equation
	\begin{equation}
	\label{Sobajima--Wakasugi eq}
	\begin{cases}
		\pa_tv^*-D\pa_y^2v^*=0,&y\in\R,\ t>0,\\
		v^*(y,0)=\lr{y}^{-2\rho},&y\in\R,
	\end{cases}
	\end{equation}
	and positive constants $C_{\textup{S},1}^*$ and $C_{\textup{S},2}^*$ such that for every $y\in\R$ and $t\ge0$,
	\begin{equation}
	\label{prop of Sobajima--Wakasugi}
	C_{\textup{S},1}^*(1+\lr{y}^{2-\beta}+t)^{-\frac{2\rho}{2-\beta}}\le\overline{v^*}(y,t)\le C_{\textup{S},2}^*(1+\lr{y}^{2-\beta}+t)^{-\frac{2\rho}{2-\beta}}.
	\end{equation}
\end{lemma}

By Lemma \ref{supersol of Sobajima--Wakasugi}, we deduce the following upper bound for $e^{-tL}[\lr{\cdot}^{-1-\alpha}]$.

\begin{lemma}
	\label{upper bound}
	Assume that $-\frac{1}{2}<\alpha<\frac{1}{2}$ and $\delta\in\left(0,\frac{1+2\alpha}{2}\right)$.
	Then there exist a positive function $\supv\in C^2(\R\times[0,\infty))$, positive constants $C_{\textup{S},1}$ and $C_{\textup{S},2}$ such that for every $x\in\R$ and $t\ge0$,
	\begin{gather}
		\label{upper bound 1}
		\supv(x,t)\ge e^{-tL}[\lr{\cdot}^{-1-\alpha}](x),\\
		\label{upper bound 2}
		C_{\textup{S},1}\lr{x}^\alpha(1+\lr{x}^2+t)^{-\frac{1+2\alpha}{2}+\delta}\le\supv(x,t)\le C_{\textup{S},2}\lr{x}^\alpha(1+\lr{x}^2+t)^{-\frac{1+2\alpha}{2}+\delta}.
	\end{gather}
\end{lemma}

\begin{proof}
	We set $H\colon\R\to\R$ as
	\[
	H(x)=\int_0^x\psi(y)^{-2}\,dy,\quad x\in\R,
	\]
	then we see that $H$ is bijective and $|H(x)|\to\infty$ as $|x|\to\infty$.
	We define a positive function $D\in C(\R)$ by $D=\psi(\widetilde{H})^{-4}$, where $\widetilde{H}=H^{-1}$, and $\beta=-\frac{4\alpha}{1-2\alpha}<1$.
	It follows from the l'H\^{o}pital rule and \eqref{prop1 of psi} that $D$ and $\beta$ satisfy \eqref{prop of D} for some positive constant $D_0$.
	We set $\rho=\frac{1-\beta}{2}-\frac{2-\beta}{2}\delta$, then noting that $\frac{1-\beta}{2-\beta}=\frac{1+2\alpha}{2}$ and $2-\beta=\frac{2}{1-2\alpha}$, we have $\rho\in\left(0,\frac{1-\beta}{2}\right)$.
	Therefore, by Lemma \ref{supersol of Sobajima--Wakasugi}, there exists a supersolution $\overline{v^*}\in C^2(\R\times[0,\infty))$ to \eqref{Sobajima--Wakasugi eq} satisfying \eqref{prop of Sobajima--Wakasugi}.
	We set $v(t)=e^{-tL}[\lr{\cdot}^{-1-\alpha}]$ and
	\[
	v^*(y,t)=\psi(\widetilde{H}(y))^{-1}v(\widetilde{H}(y),t),\quad y\in\R,\ t\ge0,
	\]
	then we see that $v^*$ satisfies
	\[
	\begin{cases}
		\pa_tv^*-D\pa_y^2v^*=0,&y\in\R,\ t>0,\\
		v^*(y,0)\le C\lr{y}^{-2\rho},&y\in\R.
	\end{cases}
	\]
	Since $\overline{v^*}$ is a supersolution to \eqref{Sobajima--Wakasugi eq}, we obtain $v^*\le C_\text{S}^\prime\overline{v^*}$ for some positive constant $C_\text{S}^\prime$, that is, for every $x\in\R$ and $t\ge0$,
	\[
	C_\text{S}^\prime\psi(x)\overline{v^*}(H(x),t)\ge v(x,t).
	\]
	Noting that there exist positive constants $C_{\textup{S},1}^\prime,C_{\textup{S},2}^\prime$ such that $C_{\textup{S},1}^\prime\lr{x}^2\le\lr{H(x)}^{2-\beta}\le C_{\textup{S},2}^\prime\lr{x}^2$ for every $x\in\R$, by \eqref{prop of Sobajima--Wakasugi}, we get for every $x\in\R$ and $t\ge0$,
	\[
	C\lr{x}^\alpha(1+\lr{x}^2+t)^{-\frac{1+2\alpha}{2}+\delta}\le C_\text{S}^\prime\psi(x)\overline{v^*}(H(x),t)\le C\lr{x}^\alpha(1+\lr{x}^2+t)^{-\frac{1+2\alpha}{2}+\delta}.
	\]
	Setting a positive function $\supv\in C^2(\R\times[0,\infty))$ as
	\[
	\supv(x,t)=C_\text{S}^\prime\psi(x)\overline{v^*}(H(x),t),\quad x\in\R,\
	t\ge0,
	\]
	we have the desired result. 
\end{proof}

\begin{lemma}
	\label{case III}
	In the case \textbf{\textup{(\III)}}, \eqref{main eq} possesses a nontrivial global-in-time solution for some initial data.
\end{lemma}

\begin{proof}
	We fix $\delta\in\left(0,\frac{1+2\alpha}{2}\right)$ satisfying $\frac{1+\alpha}{2}(p-1)+\frac{m}{2}-1>\delta(p+1)$ and $\frac{1+2\alpha}{2}p-1>\delta(p+1)$.
	By Lemma \ref{upper bound}, there exists a positive function $\supv\in C^2(\R\times[0,\infty))$ satisfying \eqref{upper bound 1} and \eqref{upper bound 2}.
	It follows from \eqref{upper bound 2} that for every $y\in\R$ and $s\ge0$,
	\[
	\begin{split}
		&\lr{y}^{-m}\supv(y,s)^p\\
		&\le C\lr{y}^{\alpha p-m}(1+\lr{y}^2+s)^{-\frac{1+2\alpha}{2}p+\delta p}\\
		&\le C
		\begin{cases}
		\lr{y}^{\alpha p-m}(1+s)^{-\frac{1+\alpha}{2}(p-1)-\frac{m}{2}+\delta p}(1+\lr{y}^2+s)^{-\frac{1+\alpha}{2}-\frac{\alpha}{2}p+\frac{m}{2}},&\text{if}\ m\le1+\alpha+\alpha p,\\
		\lr{y}^{-1-\alpha}(1+s)^{-\frac{1+2\alpha}{2}p+\delta p},&\text{if}\ m>1+\alpha+\alpha p
		\end{cases}
		\\
		&\le C\lr{y}^{-1-\alpha}(1+s)^{-1-\delta}.
	\end{split}
	\]
	This together with \eqref{upper bound 1} implies that for every $x\in\R$ and $t>0$,
	\begin{equation}
		\label{ineq III1}
		\begin{split}
		\int_0^te^{-(t-s)L}[\lr{\cdot}^{-m}\supv(s)^p](x)\,ds&\le C\int_0^t(1+s)^{-1-\delta}e^{-(t-s)L}[\lr{\cdot}^{-1-\alpha}](x)\,ds\\
		&\le C\lr{x}^\alpha\int_0^t(1+s)^{-1-\delta}(1+\lr{x}^2+t-s)^{-\frac{1+2\alpha}{2}+\delta}\,ds.
		\end{split}
	\end{equation}
	It follows that for every $x\in\R$ and $t>0$,
	\[
		\begin{split}
		&\int_0^t(1+s)^{-1-\delta}(1+\lr{x}^2+t-s)^{-\frac{1+2\alpha}{2}+\delta}\,ds\\
		&\le
		\begin{cases}
			\int_0^t(1+s)^{-1-\delta}(1+t-s)^{-\frac{1+2\alpha}{2}+\delta}\,ds,&\text{if}\ |x|\le\sqrt{t},\\
			\lr{x}^{-1-2\alpha+2\delta}\int_0^t(1+s)^{-1-\delta}\,ds,&\text{if}\ |x|>\sqrt{t}
		\end{cases}
		\\
		&\le C
		\begin{cases}
			(1+t)^{-\frac{1+2\alpha}{2}+\delta},&\text{if}\ |x|\le\sqrt{t},\\
			\lr{x}^{-1-2\alpha+2\delta},&\text{if}\ |x|>\sqrt{t}
		\end{cases}
		\\
		&\le C(1+\lr{x}^2+t)^{-\frac{1+2\alpha}{2}+\delta}.
		\end{split}
	\]
	Combining the above inequality with \eqref{ineq III1} and \eqref{upper bound 2}, we have for every $x\in\R$ and $t>0$,
	\begin{equation}
		\label{ineq III3}
		\int_0^te^{-(t-s)L}[\lr{\cdot}^{-m}\supv(s)^p](x)\,ds\le C\lr{x}^\alpha(1+\lr{x}^2+t)^{-\frac{1+2\alpha}{2}+\delta}\le C\supv(x,t).
	\end{equation}
	We fix a sufficiently small constant $\ep>0$ and set a positive function $\supu\in BC(\R\times(0,\infty))$ as $\supu=\ep\supv$.
	By \eqref{upper bound 1} and \eqref{ineq III3}, we obtain for every $x\in\R$ and $t>0$,
	\[
	\begin{split}
		e^{-tL}\left[\frac{\ep}{2}\lr{\cdot}^{-\alpha-1}\right](x)+\int_0^te^{-(t-s)L}[\lr{\cdot}^{-m}\supu(s)^p](x)\,ds&\le\frac{1}{2}\supu(x,t)+C\ep^{p-1}\supu(x,t)\\
		&\le\supu(x,t).
	\end{split}
	\]
	Therefore, by Lemma \ref{supersol method}, we see that \eqref{main eq} possesses a nontrivial global-in-time solution.
	The proof is complete.
\end{proof}

Finally, we deduce the desired result.

\begin{proof}[\textbf{\textup{Proof of Theorem \ref{main thm} (ii)}}]
	Combining Lemma \ref{case I}, Lemma \ref{case II} and Lemma \ref{case III}, we have the desired result.
\end{proof}

\section{Nonexistence of nontrivial global-in-time solutions}
\label{section4}
In this section, we prove Theorem \ref{main thm} \textbf{(i)}.
To this aim, we consider the following four cases (cf. Table \ref{table2}):
\begin{enumerate}
	\item[\textbf{(A)}] $\alpha\ge\alpha_*(1)$ and $1<p\le1+\frac{[2-m]_+}{1+\alpha}$;
	\item[\textbf{(B)}] $-\frac{1}{2}<\alpha<\frac{1}{2}$ and $1<p<\frac{2}{1+2\alpha}$;
	\item[\textbf{(C)}] $\alpha\le-\frac{1}{2}$ and $p>1$;
	\item[\textbf{(D)}] $-\frac{1}{2}<\alpha<\frac{1}{2}$ and $p=\frac{2}{1+2\alpha}$.
\end{enumerate}
We start by describing the definition and the property of the test function.
We fix two functions $\eta\in C^\infty([0,\infty))$ and $\eta^*\in L^\infty(0,\infty)$ satisfying
\[
\eta(s)
\begin{cases}
	=1,&\text{if}\ 0\le s\le\frac{1}{2},\\
	\text{is decreasing},&\text{if}\ \frac{1}{2}<s<1,\\
	=0,&\text{if}\ s\ge1,
\end{cases}
\quad\eta^*(s)=
\begin{cases}
	0,&\text{if}\ 0\le s\le\frac{1}{2},\\
	\eta(s),&\text{if}\ s>\frac{1}{2},
\end{cases}
\]
and set for $R>0$, $x\in\R$ and $t\ge0$,
\[
\xi_R(x,t)=\frac{|x|^2+t}{R},\quad\Phi_R(x,t)=\eta(\xi_R(x,t))^{2p^\prime},\quad\Phi_R^*(x,t)=\eta^*(\xi_R(x,t))^{2p^\prime},
\]
where $p^\prime=\frac{p}{p-1}$ is the H\"{o}lder conjugate of $p$.
In the proof of Theorem \ref{main thm} \textbf{(i)}, we use the function $\psi\Phi_R$ as a test function.
The following lemma describes the property of the test function used in the test function method.

\begin{lemma}
	\label{prop of test func}
	For every $R>0$, there exists a positive constant $C_\textup{T}$ such that
	\begin{equation}
		|\pa_t(\psi\Phi_R)|+|L(\psi\Phi_R)|\le C_\textup{T}R^{-1}\psi[\Phi_R^*]^\frac{1}{p}\quad\text{in}\ \R\times[0,\infty).
	\end{equation}
\end{lemma}

\begin{proof}
	By simple calculation, we have
	\[
	\begin{split}
		|\pa_t(\psi\Phi_R)|\le2p^\prime\|\eta\eta^\prime\|_{L^\infty}R^{-1}\psi\eta^*(\xi_R)^{2p^\prime-2}&\quad\text{in}\ \R\times[0,\infty),\\
		|L(\psi\Phi_R)|\le C_\textup{T}^\prime R^{-1}\psi\eta^*(\xi_R)^{2p^\prime-2}&\quad\text{in}\ \R\times[0,\infty),
	\end{split}	
	\]
	where $C_\textup{T}^\prime=8p^\prime\left\|\eta\eta^\prime\frac{x\psi^\prime}{\psi}\right\|_{L^\infty}+4p^\prime\|\eta\eta^\prime\|_{L^\infty}+8p^\prime(2p^\prime-1)\|\eta\|_{L^\infty}^2+4p^\prime\|\eta\eta^{\prime\prime}\|_{L^\infty}$.
	The proof is complete.
\end{proof}

First, we prove the case \textbf{(A)}.

\begin{lemma}
	\label{case A}
	In the case \textbf{\textup{(A)}}, \eqref{main eq} does not have nontrivial global-in-time solutions.
\end{lemma}

\begin{proof}
	Assume that there exists a global-in-time solution to \eqref{main eq} for some initial data $u_0\neq0$.
	Since $u_0$ is continuous and nonnegative, we can choose $\ep>0$ and $R_0>0$ such that for every $R\ge R_0$,
	\[
	\int_\R u_0\Phi_R(\cdot,0)\,dx\ge\ep.
	\]
	We fix $R\ge R_0$.
	Multiplying $\psi\Phi_R$ to the equation in \eqref{main eq} and integrating it over $\R\times[0,\infty)$, we see from integration by parts and Lemma \ref{prop of test func} that
	\begin{equation}
		\label{ineq A1}
		\begin{split}
			&\int_0^\infty\int_\R\lr{x}^{-m}u^p\psi\Phi_R\,dx\,dt\\
			&=\int_0^\infty\int_\R(\pa_tu-\pa_x^2u+Vu)\psi\Phi_R\,dx\,dt\\
			&=\int_0^\infty\frac{d}{dt}\left(\int_\R u\psi\Phi_R\,dx\right)\,dt+\int_0^\infty\int_\R u(-\pa_t(\psi\Phi_R)+L(\psi\Phi_R))\,dx\,dt\\
			&\le-\int_\R u_0\Phi_R(\cdot,0)\,dx+CR^{-1}\int_0^\infty\int_\R u\psi[\Phi_R^*]^\frac{1}{p}\,dx\,dt.
		\end{split}
	\end{equation}
	By the H\"{o}lder inequality, we have
	\[
		\begin{split}
			&R^{-1}\int_0^\infty\int_\R u\psi[\Phi_R^*]^{\frac{1}{p}}\,dx\,dt\\
			&\le R^{-1}\left(\int_0^R\int_{-\sqrt{R}}^{\sqrt{R}}\lr{x}^{\frac{mp^\prime}{p}}\psi\,dx\,dt\right)^{\frac{1}{p^\prime}}\left(\int_0^\infty\int_\R\lr{x}^{-m}u^p\psi\Phi_R^*\,dx\,dt\right)^{\frac{1}{p}}\\
			&\le CR^{-1+\frac{3+\alpha}{2p^\prime}+\frac{m}{2p}}\left(\int_0^\infty\int_\R\lr{x}^{-m}u^p\psi\Phi_R^*\,dx\,dt\right)^{\frac{1}{p}}\\
			&\le C\left(\int_0^\infty\int_\R\lr{x}^{-m}u^p\psi\Phi_R^*\,dx\,dt\right)^{\frac{1}{p}}.
		\end{split}
	\]
	Combining the above inequality with \eqref{ineq A1}, we obtain
	\[
	\int_\R u_0\Phi_R(\cdot,0)\,dx+\int_0^\infty\int_\R\lr{x}^{-m}u^p\psi\Phi_R\,dx\,dt\le C\left(\int_0^\infty\int_\R\lr{x}^{-m}u^p\psi\Phi_R^*\,dx\,dt\right)^{\frac{1}{p}}.
	\]
	This together with the Fatou lemma implies that
	\[
	\int_0^\infty\int_\R\lr{x}^{-m}u^p\psi\,dx\,dt\le\liminf_{R\to\infty}\int_0^\infty\int_\R\lr{x}^{-m}u^p\psi\Phi_R\,dx\,dt<\infty.
	\]
	Therefore, by the dominated convergence theorem, we get
	\[
	0<\ep\le\liminf_{R\to\infty}\int_\R u_0\Phi_R(\cdot,0)\,dx\le C\lim_{R\to\infty}\left(\int_0^\infty\int_\R\lr{x}^{-m}u^p\psi\Phi_R^*\,dx\,dt\right)^{\frac{1}{p}}=0,
	\]
	which is a contradiction.
	The proof is complete.
\end{proof}

Here we consider the cases \textbf{(B)}, \textbf{(C)} and \textbf{(D)}.
Without loss of generality, we may assume that $u_0(0)>0$.
By the comparison principle, it suffices to prove the global nonexistence of the following Cauchy problem:
\begin{equation}
	\label{main eq alt}
	\begin{cases}
		\pa_tu-\pa_x^2u+Vu=au^p,&x\in\R,\ t>0,\\
		u(x,0)=u_0^-(x)\ge0,&x\in\R,
	\end{cases}
\end{equation}
where $a,u_0^-\in C_0^\infty(\R)\backslash\{0\}$ are even functions which satisfy for every $x\in\R$,
\begin{align}
	\label{prop of a}
	0\le a(x)\le\lr{x}^{-m},&\quad x(a(x)\psi(x)^{p-1})^{\prime}\le0,\\
	\label{prop of u_0^-}
	0\le u_0^-(x)\le u_0(x),&\quad x(\psi(x)^{-1}u_0^-(x))^{\prime}\le0.
\end{align}
The following lemma describes a modified version of the method used in \cite{P1997} and plays an important role in the proof of the cases \textbf{(B)}, \textbf{(C)} and \textbf{(D)}.

\begin{lemma}
	\label{Pinsky method}
	Assume that $u$ is a global-in-time solution to \eqref{main eq alt}.
	We set $u_*=\psi^{-1}u$.
	Then, for every $t\ge0$, $u_*(t)$ is nonincreasing for $x\in[0,\infty)$ and nondecreasing for $x\in(-\infty,0]$.
	Similarly, for every $t\ge0$, $v_*(t)=\psi^{-1}(e^{-tL}u_0^-)$ is also nonincreasing for $x\in[0,\infty)$ and nondecreasing for $x\in(-\infty,0]$.
\end{lemma}

\begin{proof}
	We see that $u_*$ is a solution to the following Cauchy problem:
	\begin{equation}
		\label{eq of u_*}
		\begin{cases}
			\pa_tu_*-\psi^{-2}\pa_x(\psi^2\pa_xu_*)=a\psi^{p-1}u_*^p,&x\in\R,\ t>0,\\
			u_*(x,0)=\psi(x)^{-1}u_0^-(x)\ge0,&x\in\R.
		\end{cases}
	\end{equation}
	By the parabolic regularity theorem, we see that $u_*$ is a $C^3(\R)$-function with respect to $x$.
	We set $w=\psi^2\pa_xu_*$, then differentiating \eqref{eq of u_*} with respect to $x$ and multiplying $\psi^2$ to it, we have
	\begin{equation}
		\label{eq of w}
		\pa_tw-\psi^2\pa_x(\psi^{-2}\pa_xw)=\psi^2(a\psi^{p-1})^\prime u_*^p+pa\psi^{p-1}u_*^{p-1}w,\quad x\in\R,\ t>0.
	\end{equation}
	Noting that $V$, $a$ and $u_0^-$ are even, by the uniqueness of the solution to \eqref{main eq alt}, we see that $u$ is also even with respect to $x$.
	Therefore we see that $w$ is odd with respect to $x$.
	This together with \eqref{prop of u_0^-} implies that
	\[
	\begin{cases}
		w(0,t)=0,&t\ge0,\\
		w(x,0)\le0,&x\ge0.
	\end{cases}
	\]
	Thus, noting that $\psi$ is positive and $\sup_{(x,t)\in\R\times(0,T)}w<\infty$ for every $T>0$, and combining \eqref{prop of a}, \eqref{eq of w} and the comparison principle, we obtain
	\[
	x\pa_xu_*(x,t)\le0,\quad x\ge0,\ t\ge0.
	\]
	Since $u_*$ is even with respect to $x$, the proof is complete.
\end{proof}

Next, we prove the cases \textbf{(B)} and \textbf{(C)}.

\begin{lemma}
	\label{case B and C}
	In the cases \textbf{\textup{(B)}} and \textbf{\textup{(C)}}, \eqref{main eq} does not have nontrivial global-in-time solutions.
\end{lemma}

\begin{proof}
	Assume that there exists a global-in-time solution to \eqref{main eq alt}.
	We can choose $R_0\ge1$ such that $\supp u_0^-\subset\left[-\sqrt{\frac{R_0}{2}},\sqrt{\frac{R_0}{2}}\right]$.
	We fix $\delta\in(0,1)$ satisfying $a(x)\ge\frac{a(0)}{2}$ for every $x\in[-\delta,\delta]$.
	We fix $R\ge R_0$.
	Similarly to the proof of Lemma \ref{case A}, we have
	\begin{equation}
		\label{ineq BC1}
		\int_\R u_0^-\psi\,dx+\int_0^\infty\int_\R au^p\psi\Phi_R\,dx\,dt\le CR^{-1}\int_0^\infty\int_\R u\psi[\Phi_R^*]^{\frac{1}{p}}\,dx\,dt.
	\end{equation}
	We set $u_*=\psi^{-1}u$.
	By the H\"{o}lder inequality, we obtain
	\begin{equation}
		\label{ineq BC2}
		\begin{split}
			R^{-1}\int_0^\infty\int_\R u\psi[\Phi_R^*]^{\frac{1}{p}}\,dx\,dt&\le R^{-1}\left(\int_0^R\int_{-\sqrt{R}}^{\sqrt{R}}\psi^{2p^\prime}\,dx\,dt\right)^{\frac{1}{p^\prime}}\left(\int_0^\infty\int_\R u_*^p\Phi_R^*\,dx\,dt\right)^{\frac{1}{p}}\\
			&\le Cq(R)\left(\int_0^\infty\int_\R u_*^p\Phi_R^*\,dx\,dt\right)^{\frac{1}{p}},
		\end{split}
	\end{equation}
	where
	\[
	q(R)=
	\begin{cases}
		R^{-\frac{1}{p}},&\text{if}\ 1<p<\frac{1}{1+2\alpha}\ \text{or}\ \alpha\le-\frac{1}{2},\\
		R^{-\frac{1}{p}}(\log R)^{\frac{1}{p^\prime}},&\text{if}\ p=\frac{1}{1+2\alpha},\\
		R^{-1+\frac{3}{2p^\prime}+\alpha},&\text{if}\ \frac{1}{1+2\alpha}<p<\frac{2}{1+2\alpha}.
	\end{cases}
	\]
	Changing of variables $y=\frac{\delta}{\sqrt{R}}x$, by Lemma \ref{Pinsky method}, we get
	\begin{equation}
		\label{ineq BC3}
		\begin{split}
\int_0^\infty\int_\R u_*^p\Phi_R^*\,dx\,dt&=\frac{\sqrt{R}}{\delta}\int_0^\infty\int_{-\delta}^\delta u_*\left(\frac{\sqrt{R}}{\delta}y,t\right)^p\Phi_R^*\left(\frac{\sqrt{R}}{\delta}y,t\right)\,dy\,dt\\
			&\le C\sqrt{R}\int_0^\infty\int_{-\delta}^\delta au_*^p\Phi_R\,dy\,dt\\
			&\le C\sqrt{R}\int_0^\infty\int_{-\delta}^\delta au^p\psi\Phi_R\,dy\,dt\\
			&\le C\sqrt{R}\int_0^\infty\int_\R au^p\psi\Phi_R\,dx\,dt.
		\end{split}
	\end{equation}
	Combining \eqref{ineq BC1}, \eqref{ineq BC2} and \eqref{ineq BC3}, we see that
	\begin{equation}
		\label{ineq BC4}
		\int_\R u_0^-\psi\,dx+\int_0^\infty\int_\R au^p\psi\Phi_R\,dx\,dt\le CR^{\frac{1}{2p}}q(R)\left(\int_0^\infty\int_\R au^p\psi\Phi_R\,dx\,dt\right)^{\frac{1}{p}}.
	\end{equation}
	Noting that $R^\frac{1}{2p}q(R)\to0$ as $R\to\infty$, by \eqref{ineq BC4}, we have
	\[
	\lim_{R\to\infty}\int_0^\infty\int_\R au^p\psi\Phi_R\,dx\,dt=0.
	\]
	This together with \eqref{ineq BC4} implies that
	\[
	0<\int_\R u_0^-\psi\,dx\le0,
	\]
	which is a contradiction.
	The proof is complete.
\end{proof}

Next, we prove the case \textbf{(D)}.
The following lemma is used in the proof of the case \textbf{(D)} together with Lemma \ref{Pinsky method} (a similar treatment can be found in \cite{S2024}).

\begin{lemma}
	\label{optimal decay}
	Assume that $\alpha>-\frac{1}{2}$ and $f\in L^{1,\alpha}$.
	We set $v(t)=e^{-tL}f$ and $v_*=\psi^{-1}v$.
	If
	\[
	\int_\R f\psi\,dx\neq0,
	\]
	then there exists a positive constant $C_\textup{L}$ such that for every $t\ge0$,
	\begin{equation}
		\int_0^t\|v_*\|_{L^\infty}^{\frac{2}{1+2\alpha}}\,ds\ge C_\textup{L}\log(1+t).
	\end{equation}
\end{lemma}

\begin{proof}
	We only need to consider the case
	\[
	\int_\R f\psi\,dx>0.
	\]
	We denote $\chi$ as the indicator function of the interval $\left(\frac{1}{2},1\right)$.
	By the dominated convergence theorem, we can choose $R_0>0$ such that for every $R\ge R_0$,
	\[
	\int_\R f\psi\Phi_R(\cdot,0)\,dx\ge\frac{1}{2}\int_\R f\psi\,dx.
	\]
	We fix $R\ge R_0$.
	By Lemma \ref{prop of test func}, we have
	\begin{equation}
		\label{ineq OD1}
		\begin{split}
			0&=\int_0^\infty\int_\R(\pa_tv+Lv)\psi\Phi_R\,dx\,dt\\
			&=\int_0^\infty\frac{d}{dt}\left(\int_\R v\psi\Phi_R\,dx\right)\,dt+\int_0^\infty\int_\R v(-\pa_t(\psi\Phi_R)+L(\psi\Phi_R))\,dx\,dt\\
			&\le-\int_\R f\psi\Phi_R(\cdot,0)\,dx+CR^{-1}\int_0^\infty\int_\R|v|\psi[\Phi_R^*]^{\frac{1}{p}}\,dx\,dt\\
			&\le-\frac{1}{2}\int_\R f\psi\,dx+CR^{-1}\int_0^\infty\int_\R|v|\psi\chi(\xi_R)\,dx\,dt.
		\end{split}
	\end{equation}
	By the H\"{o}lder inequality, we obtain
	\[
		\begin{split}
			&R^{-1}\int_0^\infty\int_\R|v|\psi\chi(\xi_R)\,dx\,dt\\
			&\le R^{-1}\left(\int_0^\infty\int_\R\psi^2\chi(\xi_R)\,dx\,dt\right)^{\frac{2}{3+2\alpha}}\left(\int_0^\infty\int_\R|v_*|^{\frac{2}{1+2\alpha}}|v|\psi\chi(\xi_R)\,dx\,dt\right)^{\frac{1+2\alpha}{3+2\alpha}}\\
			&\le C\left(\int_0^\infty\int_\R|v_*|^{\frac{2}{1+2\alpha}}|v|\psi\chi(\xi_R)\,dx\,dt\right)^{\frac{1+2\alpha}{3+2\alpha}}.
		\end{split}
	\]
	Combining the above inequality with \eqref{ineq OD1}, we get
	\begin{equation}
		\label{ineq OD3}
		\left(\frac{1}{2}\int_\R f\psi\,dx\right)^{\frac{3+2\alpha}{1+2\alpha}}\le C\int_0^\infty\int_\R|v_*|^{\frac{2}{1+2\alpha}}|v|\psi\chi(\xi_R)\,dx\,dt.
	\end{equation}
	We define an auxiliary function $Y\colon[R_0,\infty)\to[0,\infty)$ by
	\[
	Y(R)=\int_0^R\int_0^\infty\int_\R|v_*|^{\frac{2}{1+2\alpha}}|v|\psi\chi(\xi_r)\,dx\,dt\,\frac{dr}{r},\quad R\ge R_0,
	\]
	then, by \eqref{ineq OD3}, we see that $R^{-1}\le CY'(R)$.
	This implies that
	\begin{equation}
		\label{ineq OD4}
		Y(R_0)+\log\left(\frac{R}{R_0}\right)\le CY(R).
	\end{equation}
	On the other hand, we have
	\[
	\begin{split}
		Y(R)&=\int_0^\infty\int_\R|v_*|^{\frac{2}{1+2\alpha}}|v|\psi\int_0^R\chi(\xi_r)\,\frac{dr}{r}\,dx\,dt\\
		&=\int_0^\infty\int_\R|v_*|^{\frac{2}{1+2\alpha}}|v|\psi\int_{\xi_R}^\infty\chi(s)\,\frac{ds}{s}\,dx\,dt\\
		&\le\log2\int_0^R\int_\R|v_*|^{\frac{2}{1+2\alpha}}|v|\psi\,dx\,dt\\
		&\le\log2\int_0^R\|v_*\|_{L^\infty}^{\frac{2}{1+2\alpha}}\|\psi v\|_{L^1}\,dt.
	\end{split}
	\]
	This together with Lemma \ref{contr and posi} \textbf{(i)}, we obtain
	\begin{equation}
		\label{ineq OD5}
		Y(R)\le C\int_0^R\|v_*\|_{L^\infty}^{\frac{2}{1+2\alpha}}\,dt.
	\end{equation}
	Combining \eqref{ineq OD4} and \eqref{ineq OD5}, we have the desired result.
\end{proof}

\begin{lemma}
	\label{case D}
	In the case \textbf{\textup{(D)}}, \eqref{main eq} does not have nontrivial global-in-time solutions.
\end{lemma}

\begin{proof}
	Assume that there exists a global-in-time solution to \eqref{main eq alt}.
	We set $v(t)=e^{-tL}u_0^-$ and $v_*=\psi^{-1}v$.
	Let $R_0\ge1$ and $\delta\in(0,1)$ be the same as in the proof of Lemma \ref{case B and C}.
	Similarly to the proof of Lemma \ref{case B and C}, we have for every $R\ge R_0$,
	\[
	\int_\R u_0^-\psi\,dx+\int_0^\infty\int_\R au^p\psi\Phi_R\,dx\,dt\le C\left(\int_0^\infty\int_\R au^p\psi\Phi_R\,dx\,dt\right)^\frac{1}{p}.
	\]
	This together with the Fatou lemma implies that
	\begin{equation}
		\label{ineq D}
		\int_0^\infty\int_\R au^p\psi\,dx\,dt\le\liminf_{R\to\infty}\int_0^\infty\int_\R au^p\psi\Phi_R\,dx\,dt<\infty.
	\end{equation}
	By the Harnack inequality, we obtain for every $t>0$,
	\begin{equation}
		\label{Harnack}
		\sup_{Q_-(t)}v\le C\inf_{Q_+(t)}v,
	\end{equation}
	where
	\[
	Q_-(t)=(-\delta,\delta)\times\left(t+\frac{2}{3}\delta^2,t+\frac{4}{3}\delta^2\right)
	\]
	and
	\[
	Q_+(t)=(-\delta,\delta)\times\left(t+\frac{8}{3}\delta^2,t+4\delta^2\right).
	\]
	Combining \eqref{Harnack} and Lemma \ref{Pinsky method}, we get
	\[
	\begin{split}
		\int_0^\infty\int_\R au^p\psi\,dx\,dt&\ge C\int_{\frac{10}{3}\delta^2}^\infty\int_{-\delta}^\delta v^p\,dx\,dt\\
		&\ge C\int_0^\infty\int_{-\delta}^\delta\left(\inf_{Q_+(t)}v\right)^p\,dx\,dt\\
		&\ge C\int_0^\infty\int_{-\delta}^\delta\left(\sup_{Q_-(t)}v\right)^p\,dx\,dt\\
		&\ge C\int_{\delta^2}^\infty v_*(0,t)^p\,dt\\
		&=C\int_{\delta^2}^\infty\|v_*\|_{L^\infty}^p\,dt.
	\end{split}
	\]
	By Lemma \ref{optimal decay}, this contradicts \eqref{ineq D}.
	The proof is complete.
\end{proof}

Finally, we deduce the desired result.

\begin{proof}[\textbf{\textup{Proof of Theorem \ref{main thm} \text{(i)}.}}]
	Combining Lemma \ref{case A}, Lemma \ref{case B and C} and Lemma \ref{case D}, we have the desired result.
\end{proof}

\subsection*{Declarations}
\begin{description}
	\item[Data availability] Data sharing not applicable to this article as no datasets were generated or analysed during the current study.
	\item[Conflict of interest] The authors declare that they have no conflict of interest.
\end{description}

\end{document}